\documentclass[12pt]{article}
\usepackage{amssymb,amsbsy,amsmath,amsfonts,amssymb,amscd,amsthm}
\usepackage[T1]{fontenc}
\usepackage[utf8]{inputenc}
\usepackage{graphicx}
\usepackage[all]{xy}
\usepackage{pdfpages}
\usepackage[linktocpage]{hyperref}

\usepackage{latexsym, amssymb, amsthm}
\usepackage{dsfont}
\usepackage{color}
\usepackage{enumerate}

\newcommand{\NN}{\mathds{N}}

\newcommand{\RR}{\mathds{R}}
\newcommand{\CC}{\mathds{C}}

\newcommand{\ZZ}{\mathds{Z}}
\newcommand{\EE}{\mathds{E}}

\DeclareMathAlphabet\gothic{U}{euf}{m}{n}

\newtheorem{theorem}{Theorem}
\newtheorem{lemma}[theorem]{Lemma}
\newtheorem{proposition}[theorem]{Proposition}
\newtheorem{corollary}[theorem]{Corollary}

\newtheorem{remark}[theorem]{Remark}

\numberwithin{theorem}{section}
\numberwithin{equation}{section}

\date{}
\begin{document} 
\title{Littlewood-Paley-Stein functionals: an ${\mathcal R}$-boundedness approach}
\author{ Thomas Cometx and El Maati Ouhabaz }
\maketitle     

\begin{abstract}
Let $L = \Delta + V$ be a Schr\"odinger operator with a non-negative potential $V$ on a complete Riemannian manifold $M$. We prove that the vertical Littlewood-Paley-Stein functional associated with $L$ is bounded on $L^p(M)$ {\it if and only if} the set $\{\sqrt{t}\, \nabla e^{-tL}, \, t > 0\}$ is ${\mathcal R}$-bounded on $L^p(M)$. We also introduce and study more general functionals. 
For a sequence of functions $m_k : [0, \infty) \to \CC$, we define 
\[
H((f_k)) := \Big( \sum_k \int_0^\infty | \nabla m_k(tL) f_k |^2 dt \Big)^{1/2} +   \Big( \sum_k \int_0^\infty | \sqrt{V} m_k(tL) f_k |^2 dt \Big)^{1/2}.
 \]
Under fairly reasonable  assumptions on $M$ we prove boundedness of $H$ on $L^p(M)$ in the sense 
\[ 
\| H((f_k)) \|_p \le C\,  \Big\| \Big( \sum_k |f_k|^2 \Big)^{1/2} \Big\|_p
\]
for some constant $C$ independent of $(f_k)_k$. A lower estimate is also proved on the dual space $L^{p'}$.  We introduce and study boundedness of other Littlewood-Paley-Stein type functionals and discuss their relationships to the Riesz transform.   Several examples are  given  in the paper.
\end{abstract}

\vspace{5mm}
\noindent

\vspace{5mm}

\vspace{10mm}

\noindent
{\bf Home institution:}    \\[3mm]
Institut de Math\'ematiques de Bordeaux \\ 
Universit\'e de Bordeaux, UMR 5251,  \\ 
351, Cours de la Lib\'eration  \\
33405 Talence. France.\\
Thomas.Cometx@math.u-bordeaux.fr\\
Elmaati.Ouhabaz@math.u-bordeaux.fr \\[8mm]

\newpage

\section{Introduction}\label{sec0}
Let $M$ be a complete non-compact Riemannian manifold and denote by $\nabla$ and $\Delta$ the corresponding gradient and the (positive) Laplace-Beltrami operator, respectively. One of the classical problems in harmonic analysis on manifolds concerns the boundedness on $L^p(M)$ of the  Riesz transform $R := \nabla\Delta^{-1/2}$. By integration by parts, it is obvious that $\| \nabla u \|_2 = \|\Delta^{1/2} u\|_2$ for all $u \in W^{1,2}(M)$. Therefore, the operator $R$, initially  defined on the range of $\Delta^{1/2}$ (which is dense in $L^2(M)$) has a bounded extension to $L^2(M)$. Note that $R$ takes values in $L^2(M, TM)$ where $TM $ is the tangent space. Alternatively, the  Riesz transform may also be defined by $d\Delta^{-1/2}$ where $d$ is  the exterior derivative. In this case $R$ takes values in the $L^2$ space of differential forms of order $1$. It is a singular integral operator with a kernel which may not be smooth. For this reason it is a difficult problem to know whether $R$ extends to a bounded operator on $L^p(M)$ for some or all $p \in (1,\infty)$. This problem has been studied by several authors during the last decades. We do not give an account on the subject and we refer the reader to \cite{ACDH, Bakry, Carron, ChenMagOuh, CoulhonDuong, CoulhonCPAM, Guillarmou-Hassell} and the references therein. \\
If the Riesz transform is bounded on $L^p(M)$, then it follows immediately from the analyticity of the heat semigroup that 
\begin{equation}\label{intro-1}
 \Big\| \sqrt{t} \nabla e^{-t\Delta} \Big\|_{{\mathcal L}(L^p(M), L^p(M, TM))} \le C \ \ \forall t > 0.
 \end{equation}
A natural question is whether \eqref{intro-1} is already sufficient to obtain the boundedness of the Riesz transform. This question is still open in general and only few results in this direction are known. It was proved by P. Auscher, Th. Coulhon, X.T. Duong and S. Hofmann \cite{ACDH} that for a manifold  $M$ satisfying  the volume doubling condition and $L^2$-Poincar\'e inequalities then  \eqref{intro-1} for some $p > 2$ implies that the Riesz transform is bounded on $L^r(M)$ for $r \in (1,p)$. See also  F. Bernicot and D. Frey \cite{BF} 
and Th. Coulhon, R. Jiang, P. Koskela and A. Sikora \cite{CJKS} for  related recent results. Note that under the volume doubling property, the $L^2$-Poincar\'e inequalities are equivalent to Gaussian upper and lower bounds for the corresponding heat kernel. The sole Gaussian upper bound together with the volume doubling condition imply the boundedness of the Riesz transform on $L^p(M)$ for $p \in (1,2]$ (cf. Th. Coulhon and X.T. Duong \cite{CoulhonDuong}). 

The study of the Riesz transform is closely related to the study of the Littlewood-Paley-Stein functional
\begin{equation}\label{H-Ga}
 H_\nabla(f) :=  \Big( \int_0^\infty | \nabla e^{-t\Delta} f |^2 dt \Big)^{1/2}
 \end{equation}
or  its variant defined in terms of the Poisson semigroup $e^{-t \sqrt{\Delta}}$. It is known (see Th. Coulhon and X.T. Duong \cite{CoulhonCPAM} or Proposition \ref{prop-multiplication} below) that if $H_\nabla$ is bounded on $L^p(M)$ then \eqref{intro-1} is satisfied.  One might then ask whether 
\eqref{intro-1} is in  turn equivalent to the boundedness of $H_\nabla$. To the best of our knowledge, this question is also open in general. The starting point of  the present paper is that if we strengthen the property that $\{ \sqrt{t} \nabla e^{-t\Delta}, \ t > 0\}$ is uniformly bounded on $L^p(M)$ (i.e., \eqref{intro-1}) into
$\{ \sqrt{t}\, \nabla\, e^{-t\Delta}, \ t > 0\}$ is ${\mathcal R}$-bounded on $L^p(M)$ (Rademacher-bounded or randomized bounded) then $H_\nabla$ is bounded on $L^p(M)$. We prove that the converse is also  true. 
Recall that $\{ \sqrt{t}\, \nabla\, e^{-t\Delta}, \ t > 0\}$ is ${\mathcal R}$-bounded on $L^p(M)$ if for every $t_k > 0$, $f_k \in L^p(M)$, $k = 1, ...,n$, 
\[
\EE \Big\| \sum_{k=1}^n {\gothic r}_k \sqrt{t_k}\,\nabla\, e^{-t_k \Delta}  f_k \Big\|_p \le C\, \EE \Big\| \sum_{k=1}^n {\gothic r}_k f_k \Big\|_p
\]
with a constant $C > 0$ independent of $t_k, f_k$ and $n$. Here, $({\gothic r}_k)_k$ is a sequence of independent Rademacher variables. \\
Actually we deal with  more general versions of the Littlewood-Paley-Stein functional and we also consider  Schr\"odinger operators $L= \Delta +V$ instead of the sole Laplacian. We introduce the functional 
\begin{equation}\label{intro-2}
 H_\Gamma ((f_k)) := \Big( \sum_k \int_0^\infty | \Gamma  m_k(tL) F(tL) f_k |^2\,  dt \Big)^{1/2}
\end{equation}
where  $m_k$ and $F$ are bounded holomorphic functions on a  sector $\Sigma( \omega_p)$ of the right half-plane with some angle $\omega_p$ and $\Gamma$ is either $\nabla$ or multiplication by $\sqrt{V}$. 
We prove that 
\begin{center}
{$a-$ the Riesz transform $\Gamma L^{-1/2} $ is bounded on $L^p$,\\
$\Downarrow$ \\
$b- \{ \sqrt{t}\, \Gamma \, e^{-t L}, \ t > 0\}$ is ${\mathcal R}$-bounded on $L^p$,\\
$\Updownarrow$ \\
$c-$ the Littlewood-Paley-Stein functional $H_\Gamma$ in \eqref{intro-2} is bounded on $L^p$.}
 \end{center}
We do not need  the general form of the square function \eqref{intro-2}  for the implication $c \Rightarrow b$, see Theorem \ref{prop-LPSR}.  We do not know  whether $b \Rightarrow a$ is true in general but we hope that putting into play the ${\mathcal R}$-boundedness idea will  shed some new light  on the problem of  boundedness of the Riesz transform.

Before describing in a more explicit way some  other contributions in this paper we recall some known results on $H_\nabla$.  A classical result of E.M. Stein \cite[Chapter IV]{Stein2} states that $H_\nabla$ is bounded $L^p(\RR^N)$ for all $p \in (1, \infty)$. This was extended to the case of sub-Laplacians on Lie groups in \cite{Stein}. On Riemannian manifolds, the boundedness on $L^p(M)$ was also considered. N. Lohou\'e \cite{Lohoue} proved several results in the setting of Cartan-Hadamard manifolds. See also J.C. Chen  \cite{Ji-Cheng}. For $p \in (1,2]$, the method  of Stein works in the general setting of any complete Riemannian manifold as pointed out by Th. Coulhon, X.T. Duong and X.D. Li  in \cite{CDL-Studia}. More precisely, it is proved there that $H_\nabla$ is bounded on 
$L^p(M)$ for all $p \in (1, 2]$ and if in addition the manifold satisfies the doubling condition \eqref{doubl} and a Gaussian upper bound \eqref{gauss}
for the corresponding heat kernel then $H_\nabla$ is of weak type $(1,1)$. We also refer to \cite{CDL-Studia} for references to other related works. These questions are also studied for elliptic operators in divergence form, we refer to the work of P. Auscher, S. Hofmann and J.M. Martell \cite{Auscher-Hofmann} for recent advance and references. For a given Schr\"odinger operator 
$L = \Delta + V$ with a non-negative potential $V \in L^1_{loc}(M)$, the method of Stein can be used to  prove  that the functional
\[ 
H(f) := \Big(  \int_0^\infty | \nabla e^{-tL} f |^2 dt \Big)^{1/2} + \Big(  \int_0^\infty | \sqrt{V} e^{-tL} f |^2 dt \Big)^{1/2}
\]
is bounded on $L^p(M)$ for all $p \in (1,2]$. See E.M. Ouhabaz \cite{Ouhabaz}. The situation is different for $p > 2$ and negative results, even for $M = \RR^N$, are given in \cite{Ouhabaz}.\\
We mention  that there are  the so-called horizontal Littlewood-Paley-Stein functionals. These functionals are of the form $\Big( \int_0^\infty | \varphi(tL) f|^2\, \frac{dt}{t} \Big)^{1/2}$ for a bounded holomorphic function $\varphi$ in a certain sector of $\CC^+$. 
They do not involve the gradient term or multiplication by $\sqrt{V}$.  Such functionals are easier to handle and their boundedness on $L^p$ can be obtained from the  bounded holomorphic functional calculus. See M. Cowling, I. Doust, A. McIntosh and A. Yagi \cite{CDMY} or Ch. Le Merdy \cite{LeMerdy} and the references therein.  

In the present paper we prove in a general setting that for all $p \in (1, 2]$ and  $F$ such that $| F(z) | \le \frac{C}{|z|^\delta}$ as $z \to \infty$ and $| F'(z) | \le \frac{C}{|z|^{1-\epsilon}}$ as $z \to 0$ for some $\delta > \frac{1}{2}$ and $\epsilon > 0$, then $H_\Gamma$ is bounded on $L^p(M)$ in the sense that there exists a constant 
$C > 0$, independent of $(f_k)$,  such that 
\begin{equation}\label{eq0} 
\| H_\Gamma((f_k)) \|_p \le C\, \sup_k \| m_k \|_{H^\infty(\Sigma(\omega_p))} \big\| \Big( \sum_k |f_k|^2 \Big)^{1/2}  \Big\|_p,
\end{equation}
where 
$$ \| m_k \|_{H^\infty(\Sigma(\omega_p))} = \sup_{z \in \Sigma(\omega_p)} | m_k(z) |.$$
See Theorem \ref{thm-hol} below. The particular case $k = 1$, $m_1 = 1$ and $F(z) = e^{-z}$ corresponds to the standard Littlewood-Paley-Stein functional which we  discussed before.  This result holds for  $p \in (2, \infty)$ under the assumption that $\{ \sqrt{t}\, \nabla\, e^{-t\Delta}, \ t > 0\}$ is ${\mathcal R}$-bounded on $L^p(M)$. We also prove a similar result for the functional
 \[
 G_\Gamma ((f_k)) := \Big( \sum_k \int_0^\infty | \Gamma m_k(tL) f_k |^2\,  dt \Big)^{1/2}
 \]
 for compactly supported functions $m_k$ which belong to a certain Sobolev space (see Theorem \ref{thm-mult}). There is a standard duality argument which provides a reverse inequality on the dual space for the classical Littlewood-Paley-Stein functional.  We adapt the argument to our general setting and prove a reverse inequality in $L^q(M)$ ($\frac{1}{q} + \frac{1}{p} = 1$) for  the previous  Littlewood-Paley-Stein functionals. See Theorem \ref{thm-reverse}. 
 
 The proof of Theorem \ref{thm-hol} uses heavily the fact that $L$ has a bounded holomorphic functional calculus on $L^p(M)$ and as a consequence $L$ satisfies square function estimates. In addition, $m_k(L), k \ge 1$  is ${\mathcal R}$-bounded on $L^p(M)$ by a result of N.J. Kalton and L. Weis \cite{KW}. This does not apply mutatis mutandis to the functional $G_\Gamma$. Instead we rely on a  recent result by L. Deleaval and Ch. Kriegler \cite{DK}. 
 
We introduce the local Littlewood-Paley-Stein functional and  the Littlewood-Paley-Stein functional at infinity defined respectively by
\[
H^{loc}_\Gamma(f) := \Big(  \int_0^1 | \Gamma e^{-tL} f |^2 dt \Big)^{1/2} \ {\rm and} \ H^{(\infty)}_\Gamma(f) := \Big( \int_1^\infty | \Gamma e^{-tL} f |^2 dt \Big)^{1/2}. 
\]
We study the boundedness on $L^p(M)$ of $H^{loc}_\Gamma$ (respectively,
$H^{(\infty)}_\Gamma$) and their relationship to the  local Riesz transform $R_{loc} := \Gamma( L+ I)^{-1/2}$ (respectively the Riesz transform at infinity $R_\infty := \nabla L^{-1/2} e^{-L}$).\footnote{The {\it quasi-} Riesz transforms $R_{loc}$ and $R_\infty$ were studied by L. Chen \cite{Chen2014} for the Laplace-Beltrami operator.}  For example, if $L = \Delta$ and $M$ has Ricci curvature bounded from below, then it is well known that $R_{loc}$
is bounded on $L^p(M)$ for all $p \in (1,\infty)$ (see D. Bakry \cite{Bakry}).  As a consequence we obtain that $H^{loc}_\nabla$ is bounded on $L^p(M)$ for all $p \in (1,\infty)$ and the lower bound
\[
C\, \| f \|_q \le  \| e^{-\Delta} f \|_q + \| H^{loc}(f)  \|_q 
\]
holds for all  $q \in (1, \infty)$. 

We give several examples in Section \ref{sec5} including Schr\"odinger operators on $\RR^N$ with a potential in a reverse H\"older class or  Schr\"odinger operators on manifolds. We shall see that for the connected sum $M_n := \RR^n \# \RR^n$ ($n \ge 2$) the Littlewood-Paley-Stein at infinity is not bounded on $L^p(M_n)$ for $p > n$. The fact that the Riesz transform is not bounded on $L^p(M_n)$ for $p > n$ was proved by Th. Coulhon and X.T. Duong \cite{CoulhonDuong}.

Although we focus on Schr\"odinger operators on manifolds, our results are also valid for elliptic operators on rough domains. Let $\Omega$ be an open subset of $\RR^N$ and consider on $L^2(\Omega)$ an elliptic operator $L= -div(A(x) \nabla \cdot)$ with real symmetric and bounded measurable coefficients. The operator $L$ is subject to the Dirichlet boundary conditions. Then \eqref{eq0} holds on $L^p(\Omega)$ for all $p \in (1, 2]$. As a particular case of the reverse inequality, we obtain  for $q \in [2, \infty)$
\begin{equation}\label{eq0-3}
C\, \| f \|_{L^q(\Omega)} \le  \Big\| \Big( \int_0^\infty | \nabla e^{-tL} f |^2 dt \Big)^{1/2} \Big\|_{L^q(\Omega)}
\end{equation}
and
\begin{equation}\label{eq0-4}
C\, \| f \|_{L^q(\Omega)} \le \| e^{-L} f \|_{L^q(\Omega)} + \Big\| \Big( \int_0^1 | \nabla e^{-tL} f |^2 dt \Big)^{1/2} \Big\|_{L^q(\Omega)}.
\end{equation}
We point out that no regularity assumption is required  on the domain nor on the coefficients of the operator. For another proof of \eqref{eq0-4} and related inequalities on a smooth domain, we refer to a recent paper by O. Ivanovici and F. Planchon \cite{Planchon}. 
If  $\Omega = \RR^N$, we prove that the  lower bounds  \eqref{eq0-3} and \eqref{eq0-4} are valid  for all $q \in (1, \infty)$.\\

\noindent {\bf Notation.}  We denote by $d$ the exterior derivative. We use either $\nabla L^{-1/2}$ or $dL^{-1/2}$ for the Riesz transform. We often write  $| \nabla f (x) |$ (or $| d f(x)|$) for the norm  in $T_xM$ (or in $T^*_xM$)  and we sometimes write $| \nabla f(x) |_x$ to emphasize the dependence of this norm in the point $x$. 
We use the notation $L^p(\Lambda^1T^*M) := L^p(M, T^*M)$ for the $L^p$-space of differential forms of order $1$ on $M$. For a Banach space $E$, $L^p(M, E)$ denotes the $L^p$ space of functions with values in $E$. As usual, the boundedness of the Riesz transform $\nabla L^{-1/2}$ on $L^p(M)$ means that $\nabla L^{-1/2}$, initially defined on the range of $L^{1/2}$, extends to a bounded operator from $L^p(M)$ into $L^p(M, TM)$. \\
For a given Banach space $E$, we use $\| . \|_E$ to denote its norm and the $L^p$-norm will be denoted by $\| . \|_p$ as usual. We shall use $dx$  for the Riemannian measure on  $M$.  
Finally, all inessential constants are denoted by $C, C',c...$\\

\noindent{\bf Acknowledgements.} The authors thank Christoph Kriegler for stimulating discussions and for pointing out the recent paper \cite{DK}. They  also thank Bernhard Haak for a precious help in the proof of Theorem \ref{thm-hol} and   the anonymous referee for his/her comments.\\
This research is partly supported by the ANR project RAGE,  ANR-18-CE-0012-01.

\section{Preliminary results}\label{sec1} 
This section is essentially a preparation for the next ones.   We start off by recalling some well known tools on the holomorphic functional calculus, square functions and ${\mathcal R}$-boundedness of a family of operators.

Let $\omega \in (0, \pi)$ and set 
$$\Sigma(\omega) := \{ z \in \CC, z \not=0, \, |\arg(z) | < \omega \}$$
the open sector of $\CC^+$ with angle $\omega$. We denote by $H^\infty(\Sigma(\omega))$ the set of bounded holomorphic functions on $\Sigma(\omega)$. By $H_0^\infty(\Sigma(\omega))$ we denote the subset 
$$H_0^\infty(\Sigma(\omega)) = \left\{ F \in H^\infty(\Sigma(\omega)), \exists\,  C, s > 0:\,  | F(z) | \le \frac{C |z|^s}{1+ |z|^{2s}} \, \forall z \in \Sigma(\omega) \right\}.$$
Consider a closed, densely defined  operator $A$ with dense range on a Banach space $E$ and suppose that it satisfies the basic resolvent estimate
$$ \| (\lambda I - A)^{-1} \| \le \frac{C}{ |\lambda |} \ \ \forall\, \lambda \notin \Sigma(\omega).$$
One defines the bounded operator $F(A)$ for $F \in H_0^\infty(\Sigma(\omega))$ by the standard Cauchy formula
\[ F(A) = \frac{1}{2\pi i} \int_\gamma F(z) (z I - A)^{-1} dz\]
on an appropriate contour $\gamma$. One says that $A$ has a bounded holomorphic functional calculus with angle $\omega$ if for some constant $C_\omega > 0$ 
$$ \| F(A) \|_{{\mathcal L}(E)} \le C_\Omega\, \|F \|_{H^\infty(\Sigma(\omega))} := C_\omega \sup_{z \in \Sigma(\omega)} |F(z)| $$
for all $F \in H_0^\infty(\Sigma(\omega))$. 
In this case, for every $F \in H^\infty(\Sigma(\omega))$, $F(A)$ is well defined and satisfies the same estimate as above. We refer to \cite{CDMY} for all the details. \\
One  of the most important consequences of the  holomorphic functional calculus in harmonic analysis concerns square function estimates. Set $E = L^p(X, \mu)$. For  $F \in H_0^\infty(\Sigma(\omega))$, we define for $g \in E$,
$$ \Big( \int_0^\infty | F(tA)g |^2\,  \frac{dt}{t} \Big)^{1/2}.$$
It turns out that this functional  is bounded on $E$, i.e., 
\begin{equation}\label{eq-sqf}
 \Big\|  \Big( \int_0^\infty | F(tA) g |^2\,  \frac{dt}{t} \Big)^{1/2} \Big\|_p \le C_F \| g \|_p.
 \end{equation}
We refer to \cite{CDMY} and \cite{LeMerdy}. 

Now let $L = \Delta +V$ be a Schr\"odinger operator with a non-negative $V \in L_{loc}^1(M)$. Since the semigroup $(e^{-tL})$ is sub-Markovian, $L$ has a bounded holomorphic functional calculus on $L^p(M)$ for all $p \in (1, \infty)$. This was proved by many authors  and the result had successive improvements during several decades. The most recent and general result in this direction states that  
$L$ has a bounded holomorphic functional calculus with angle $\omega_p =  \arcsin |\frac{2}{p} -1| + \epsilon$ (for any $\epsilon > 0$). We refer to  \cite{Carbonaro} for the precise statement. In particular, one has the square function estimate \eqref{eq-sqf} for 
$F \in H_0^\infty(\Sigma(\omega_p))$. \\
A well known duality argument which can be found in  \cite[p. 85]{Stein2}  shows that the reverse inequality holds on $L^q(M)$, that is for every $q \in (1, \infty)$ and $F$ as above
\begin{equation}\label{eq-sqf-rev}
 C'_F \| g \|_q  \le \Big\|  \Big( \int_0^\infty | F(tA) g |^2\,  \frac{dt}{t} \Big)^{1/2} \Big\|_q.
 \end{equation}

Recall that a subset ${\mathcal T}$ of ${\mathcal L}(L^p(M))$ is said  ${\mathcal R}$-bounded if there exists a constant $C > 0$ such that 
for every collection $T_1,.., T_n \in {\mathcal T}$ and every $f_1,..., f_n \in L^p(M)$
\begin{equation}\label{r-bb}
\EE \Big\| \sum_{k=1}^n {\gothic r}_k T_k f_k \Big\|_p \le C\, \EE \Big\| \sum_{k=1}^n {\gothic r}_k f_k \Big\|_p.
\end{equation}
Here, $({\gothic r}_k)_k$ is a sequence of independent Rademacher variables and $\EE$ is the usual expectation.  By the Kahane-Khintchine inequality, this definition can be reformulated as follows 
\begin{equation}\label{r-bb-2}
\Big\|  \Big( \sum_{k=1}^n | T_k f_k|^2 \Big)^{1/2}  \Big\|_p \le C\,  \Big\| \Big( \sum_{k=1}^n |f_k|^2 \Big)^{1/2} \Big\|_p.
\end{equation}
The notion of ${\mathcal R}$-bounded operators plays a very important role in many questions in functional analysis (cf. \cite{HytonenII}) as well as in the theory of maximal regularity for evolution equations (see \cite{Weis} or \cite{KW}).

For $L = \Delta + V$ and $\Gamma = \nabla$ or multiplication by $\sqrt{V}$, we shall use  the property that  the set 
$\{  \sqrt{t}\,  \Gamma e^{-tL}, \  \  t > 0 \}$ is  ${\mathcal R}$-bounded on $L^p(M)$. 
If $\Gamma = \nabla$,  then $\nabla e^{-tL} f(x) \in T_xM$ and hence 
$ | \nabla e^{-tL} f(x) | = | \nabla e^{-tL} f(x)) |_x$. This dependence of the norm on the point $x$ does not affect the proof that  \eqref{r-bb}, for $T_k = \sqrt{t_k} \nabla e^{-t_k L}$,  is equivalent (by the Kahane inequality) to  \eqref{r-bb-2} with $| . | = |. |_x$ in the LHS term.

\begin{proposition}\label{riesz-rbb} Let  $p \in (1, \infty)$ and suppose that the Riesz transform
$\Gamma L^{-1/2}$ is bounded on $L^p(M)$. Then the set $\{  \sqrt{t}\,  \Gamma e^{-tL}, \  \  t > 0 \}$ is ${\mathcal R}$-bounded on $L^p(M)$.
\end{proposition}
\begin{proof} Let $T_k := \sqrt{t_k} \Gamma e^{-t_k L}$ for $t_k > 0$ and $f_k \in L^p(M)$ for $k = 1,...,n$. We have
\begin{eqnarray*}
\EE \Big\| \sum_{k=1}^n {\gothic r}_k T_k f_k \Big\|_p &=& \EE  \Big\| \Gamma L^{-1/2} \sum_{k=1}^n {\gothic r}_k (t_k L)^{1/2} e^{-t_k L} f_k \Big\|_p\\
&\le& C\, \EE  \Big\|  \sum_{k=1}^n {\gothic r}_k (t_k L)^{1/2} e^{-t_k L} f_k \Big\|_p.
\end{eqnarray*}
Let $\phi_k(z) := \sqrt{t_k z} e^{-t_k z}$ and observe that  the sequence $(\phi_k)_k$ is uniformly bounded in $H^\infty(\Sigma(\omega_p))$. 
As we mentioned above, the operator $L$ has bounded holomorphic functional calculus on $L^p(M)$ with angle $\omega_p$. Therefore, 
by  \cite{KW} or \cite[Theorem~10.3.4]{HytonenII}, the set $\{ \phi_k(L), \, k \ge 1 \}$ is ${\mathcal R}$-bounded on $L^p(M)$. Using this in the previous inequality yields  \eqref{r-bb}.
\end{proof}

It is useful to notice that the  ${\mathcal R}$-boundedness of $\sqrt{t}\, \Gamma\, e^{-tL}$ can be reformulated in terms of the resolvent. More precisely,
\begin{proposition}\label{r-bb-resol}
Let $\delta' > \frac{1}{2}$. Then  the following assertions are equivalent\\
i) the set $\{\sqrt{t}\,  \Gamma\, e^{-tL}, \ t > 0  \}$ is ${\mathcal R}$-bounded on $L^p(M)$,\\
ii) the set $\{\sqrt{t}\,  \Gamma\, (I + tL)^{-\delta'}, \ t > 0  \}$ is ${\mathcal R}$-bounded on $L^p(M)$.
\end{proposition}
\begin{proof}
Suppose that $\{\sqrt{t}\,  \Gamma\, e^{-tL}, \ t > 0  \}$ is ${\mathcal R}$-bounded on $L^p(M)$ and let $\delta' > \frac{1}{2}$. By the Laplace transform
\begin{eqnarray*}
 \sqrt{t}\,  \Gamma\, (I + tL)^{-\delta'} &=& c_{\delta'} \sqrt{t} \int_0^\infty s^{\delta'-1} e^{-s} \Gamma\, e^{-stL} \,ds\\
 &=& c_{\delta'} \int_0^\infty a_t(s) \sqrt{s}\,  \Gamma\, e^{-sL} \, ds
 \end{eqnarray*}
 with $a_t(s) := t^{\frac{1}{2}-\delta'} s^{\delta'- \frac{3}{2}} e^{-s/t}$. Since $\delta' > \frac{1}{2}$ we have $\int_0^\infty a_t(s) ds = c'_{\delta'}$.  We can then apply  \cite[Lemma 3.2]{dePagter} to conclude that the set in $ii)$ is ${\mathcal R}$-bounded.\\
 Suppose now that $ii)$ is satisfied with some $\delta' > \frac{1}{2}$. Define for each $t > 0$, $\phi_t(z) := (1+ tz)^{\delta'}e^{-tz}$. Then $(\phi_t)_t$ is uniformly bounded in  $H^\infty(\Sigma(\omega_p))$. Hence, $\{ \phi_t(L), \ t > 0 \}$ is 
 ${\mathcal R}$-bounded. Taking the product of the ${\mathcal R}$-bounded operators $\sqrt{t}\, \Gamma\, (1+ tL)^{-\delta'}$ and 
 $\phi_t(L)$ gives assertion $i)$. 
\end{proof}

We finish this section by the following lemma. 
\begin{lemma}\label{lem1-0}
Let $I$ be an interval of $\RR$ and suppose that for each $t \in I$, $S_t$ is a bounded operator on $L^p(M)$ (with values in $L^p(M)$ or in $L^p(M, TM)$). Then the set $\{ S_t, \ t \in I \}$ is 
${\mathcal R}$-bounded on $L^p(M)$ {\it if and only if} there exists a constant $C > 0$ such that
$$\Big\| \Big( \int_I  | S_t u(t) |^2 \, dt  \Big)^{1/2} \Big\|_p \le C\, \Big\| \Big( \int_I  | u(t) |^2 \, dt  \Big)^{1/2} \Big\|_p$$
for all $u \in L^p(M, L^2(I))$.
\end{lemma}
This lemma is proved in  \cite{Weis} (see 4.a) in the case where  $S_t: L^p(M) \to L^p(M)$ for each $t>0$. Here $M$ is any $\sigma$-finite measured space.  In our case,  these  operators may take values in 
$L^p(M, TM)$ as in the case of $S_t = \sqrt{t} \nabla e^{-tL}$.   Here,  $|S_t u(t,x)|$ is actually $|S_t u(t,x)|_x$ where  $|.|_x$ is again the norm in the tangent space  $T_xM$ at the point $x$. For the proof one can either repeat the argument in  \cite{Weis} or argue by taking projection on each $e_j$ where $\{e_1,...,e_m \}$ is an orthonormal basis of $T_xM$. 

\section{Littlewood-Paley-Stein inequalities and ${\mathcal R}$-boundedness}

Let $L = \Delta + V$ with $0 \le V \in L^1_{loc}(M)$.  We use again the notation $\Gamma$ for the gradient $\nabla$ or the multiplication by $\sqrt{V}$. 

\begin{theorem}\label{prop-LPSR} Let $H_\Gamma(f) =  \Big( \int_0^\infty | \Gamma e^{-tL} f |^2\, dt \Big)^{1/2}$ and $p \in (1,\infty)$. Then $H_\Gamma$ is bounded on $L^p(M)$ if and only if the set $\{ \sqrt{t}\, \Gamma e^{-tL}, t > 0 \}$ is 
${\mathcal R}$-bounded on $L^p(M)$. 
\end{theorem}

\begin{proof} 

Suppose that  $H_\Gamma$ is bounded on $L^p(M)$.  We prove that $\{ \sqrt{t}\, \Gamma e^{-tL}, t > 0 \}$ is ${\mathcal R}$-bounded on $L^p(M)$. For the converse we shall prove a more  general result in the next section and hence we do not give the details here in order to avoid repetition.\\
Let $t_k \in (0, \infty)$ and $f_k  \in L^p(M)$ for $k = 1,...,N$.  We start by estimating the quantity 
$ I := \EE \left|  \sum_k  {\gothic r}_k \sqrt{t_k}\, \Gamma\,  e^{-t_k L} f_k \right|^2$. Using (twice) the independence of the Rademacher variables we have 
\begin{eqnarray*}
I &=& - \int_0^\infty \frac{d}{dt} \EE | \Gamma e^{-tL} \sum_k  {\gothic r}_k \sqrt{t_k} e^{-t_k L} f_k |^2\, dt\\
&=& 2 \int_0^\infty \EE \left[ (\Gamma e^{-tL} \sum_k  {\gothic r}_k \sqrt{t_k}  e^{-t_k L} f_k) \cdot (\Gamma e^{-tL}\sum_k  {\gothic r}_k \sqrt{t_k} L e^{-t_k L} f_k)\right] \, dt\\
&=& 2  \int_0^\infty \EE   \sum_k  \Gamma e^{-tL} {\gothic r}_k \sqrt{t_k} e^{-t_k L} f_k \cdot  \Gamma e^{-tL}   {\gothic r}_k \sqrt{t_k} L e^{-t_k L} f_k \, dt\\
&=& 2 \int_0^\infty \EE  \sum_k  \Gamma e^{-tL} {\gothic r}_k e^{-t_k L} f_k \cdot  \Gamma e^{-tL}   {\gothic r}_k (t_k L) e^{-t_k L} f_k \, dt\\
&=&2  \int_0^\infty \EE \left[ (  \Gamma e^{-tL} \sum_k {\gothic r}_k e^{-t_k L} f_k) \cdot  (\Gamma e^{-tL}  \sum_k {\gothic r}_k (t_k L) e^{-t_k L} f_k )\right] \, dt.
\end{eqnarray*}
Next, by the Cauchy-Schwarz inequality,
\begin{eqnarray*}
I &\le& 2 \int_0^\infty \Big( \EE | \Gamma e^{-tL} \sum_k {\gothic r}_k e^{-t_k L} f_k |^2 \Big)^{1/2} \Big( \EE | \Gamma e^{-tL} \sum_k {\gothic r}_k (t_k L) e^{-t_k L} f_k |^2 \Big)^{1/2}\, dt\\
&\le& \int_0^\infty  \EE | \Gamma e^{-tL} \sum_k {\gothic r}_k e^{-t_k L} f_k |^2  \, dt + \int_0^\infty  \EE | \Gamma e^{-tL} \sum_k {\gothic r}_k (t_k L) e^{-t_k L} f_k |^2  \, dt.
\end{eqnarray*}
Therefore,
\begin{equation}\label{eq10HR}
I \le \EE \left[ \Big(H_\Gamma ( \sum_k {\gothic r}_k e^{-t_k L} f_k)\Big)^2 \right]  + \EE \left[ \Big(H_\Gamma ( \sum_k {\gothic r}_k (t_k L) e^{-t_k L} f_k) \Big)^2 \right].
\end{equation}
In order to continue, we look at $H_\Gamma$ as the norm in $L^2((0, \infty), dt)$ so that 
\[
 \EE \left[ \Big(H_\Gamma ( \sum_k {\gothic r}_k e^{-t_k L} f_k)\Big)^2 \right] = \EE \Big\| \sum_k {\gothic r}_k \Gamma e^{-tL} e^{-t_k L} f_k \Big\|^2_{L^2((0,\infty), dt)}.
 \]
Hence by the Kahane inequality,
\begin{equation}\label{eq10HRp}
c_p \sqrt{I} \le  \left| \EE \left[ \Big(H_\Gamma ( \sum_k {\gothic r}_k e^{-t_k L} f_k)\Big)^p \right] \right|^{1/p}  +  \left| \EE \left[ \Big(H_\Gamma ( \sum_k {\gothic r}_k (t_k L) e^{-t_k L} f_k) \Big)^p \right] \right|^{1/p}
\end{equation}
for some constant $c_p > 0$. Now we use the assumption that $H_\Gamma$ is bounded on $L^p(M)$ and obtain 
\begin{eqnarray*}
\Big\| \sqrt{I} \Big\|_p &\le&  C \Big( \left| \EE \Big\| \sum_k {\gothic r}_k e^{-t_k L} f_k \Big\|_p^p \right|^{1/p}  + 
 \left| \EE \Big\| \sum_k {\gothic r}_k (t_k L) e^{-t_k L} f_k \Big\|_p^p \right|^{1/p} \Big)\\
 &\le& C' \Big( \EE \Big\| \sum_k {\gothic r}_k e^{-t_k L} f_k \Big\|_p + 
 \EE \Big\| \sum_k {\gothic r}_k (t_k L) e^{-t_k L} f_k \Big\|_p \Big)
\end{eqnarray*}
where we used again the Kahane inequality.  On the other hand, it is easy to see by the Kahane inequality that 
$\| \sqrt{I} \|_p$ is equivalent to $\EE \Big\|  \sum_k  {\gothic r}_k \sqrt{t_k}\,  \Gamma\, e^{-t_k L} f_k \Big\|_p$. 
Since the operator $L$ has a bounded holomorphic functional calculus on $L^p(M)$, it follows from \cite{KW} or \cite[Theorem~10.3.4]{HytonenII}  that  $(e^{-tL})_{t > 0} $ and $ (tL e^{-tL})_{t>0} $ are ${\mathcal R}$-bounded on $L^p(M)$. This  and the previous estimates give
\[
\EE \Big\|   \sum_k  {\gothic r}_k \sqrt{t_k}\,  \Gamma\,  e^{-t_k L} f_k \Big\|_p \le C\, \EE \Big\|  \sum_k  {\gothic r}_k  f_k \Big\|_p
\]
with a constant $C$ independent of $t_k$ and $f_k$. This proves that $\{ \sqrt{t}\, \Gamma e^{-tL}, t > 0 \}$ is ${\mathcal R}$-bounded on $L^p(M)$.
 
\end{proof}
We have the following corollary which is valid on  any complete Riemannian manifold $M$. 
\begin{corollary}\label{cor-LPSR} Let $p \in (1, 2]$. Then the  set $\{ \sqrt{t}\, \Gamma e^{-tL}, t > 0 \}$ is 
${\mathcal R}$-bounded on $L^p(M)$. 
\end{corollary}

\begin{proof} As already mentioned in the introduction, $H_\Gamma$ is always bounded on $L^p(M)$ for all $p \in (1,2]$ (cf. \cite{Ouhabaz} for Schr\"odinger operators and $\Gamma = \nabla$ or $\sqrt{V}$ and \cite{CDL-Studia} for $H_\nabla$ and $ L =  \Delta$).  The corollary is then a consequence of  the previous theorem.  
\end{proof}

\begin{remark}\label{rem1}
For $\Gamma = \sqrt{V}$ we have the following alternative  proof for the 
${\mathcal R}$-boundedness of $\{\sqrt{t}\, \sqrt{V}\, e^{-tL}, \ t > 0\}$ on $L^p(M)$ for $p \in (1,2]$. We have 
\begin{eqnarray*}
\int_0^t \sqrt{s}\, \sqrt{V} e^{-sL} |f | ds &\le& \frac{t}{\sqrt{2}} \Big( \int_0^t | \sqrt{V}\, e^{-sL} |f| |^2 ds \Big)^{1/2}\\
&\le&  \frac{t}{\sqrt{2}}  \Big( \int_0^\infty | \sqrt{V}\, e^{-sL} |f| |^2 ds \Big)^{1/2}.
\end{eqnarray*}
It follows from the fact that  $f \mapsto \Big( \int_0^\infty | \sqrt{V}\, e^{-sL} |f| |^2 ds \Big)^{1/2}$ is bounded on $L^p(M)$ for $p \in (1,2]$ that 
\[
\Big\| \sup_{t > 0}  \frac{1}{t} \int_0^t \sqrt{s}\, \sqrt{V}\, e^{-sL} |f | ds \Big\|_p \le C\, \| f \|_p.
\]
From this, the positivity of  $\sqrt{s}\, \sqrt{V}\, e^{-sL}$ and \cite{Weis} (4.c) it follows  that $\{\sqrt{t}\, \sqrt{V}\, e^{-tL}, \ t > 0\}$ is ${\mathcal R}$-bounded.  
\end{remark}

\section{Generalized Littlewood-Paley-Stein functionals}\label{sec2}


In this section we prove  new  Littlewood-Paley-Stein inequalities for 
$L = \Delta + V$. The first inequality  involves the holomorphic functional calculus of $L$ on $L^p(M)$ and the second one spectral multipliers with compactly supported functions.

We have already mentioned and used  that $L$ has a bounded holomorphic functional calculus with angle
$\omega_p \in (\arcsin |\frac{2}{p} -1|, \frac{\pi}{2})$ on $L^p(M)$ for $p \in (1, \infty)$. In particular, $F(L)$ is a bounded operator on $L^p(M)$ for 
$F   \in H^\infty(\Sigma(\omega_p))$. Let again $\Gamma $ be  either $\nabla$ or multiplication by $\sqrt{V}$. Our first result is the following.  

\begin{theorem}\label{thm-hol} 
Let $m_k, F  \in H^\infty(\Sigma(\omega_p))$ for $k=1, 2,...$ and assume that for some $\delta > \frac{1}{2}$ and $\epsilon_0 \in (0,1)$,  $| F(z) | \le \frac{C}{|z|^\delta}$ as $z \to \infty$, $| F'(z) | \le \frac{C}{|z|^{1-\epsilon_0}}$ as $z \to 0$
and $| m_k'(z) | \le \frac{C}{|z|^{1-\epsilon_0}}$ as $z \to 0$. 
Set 
$$ M_0 = \sup_{k} \| m_k \|_{H^\infty(\Sigma(w_p))} \quad {\rm and} \quad  M_1 = \sup_{k} \| z \mapsto z^{1-\epsilon_0} m_k'(z) \|_{H^\infty(\Sigma(w_p))}.$$ 
1) Given $p \in (1, 2]$. 
Then there exists a constant $C_F > 0$, independent of $m_k$,  such that for all $f_k \in L^p(M)$,
\begin{equation}\label{eq2-1}
\Big\|  \Big( \sum_k \int_0^\infty | \Gamma\,  m_k(tL) F(tL) f_k |^2 dt \Big)^{1/2} \Big\|_p \le C_F \Big[ M_0 + M1 \Big]   \Big\| \Big( \sum_k |f_k|^2 \Big)^{1/2} \Big\|_p.
\end{equation}
In particular, the functional 
\[
H^F_\Gamma (f) :=  \Big(  \int_0^\infty | \Gamma\, F(tL) f |^2 dt \Big)^{1/2} 
\]
is bounded on $ L^p(M)$. \\
2) If $p \in (2, \infty)$ we assume in addition that $\{ \sqrt{t}\, \Gamma\, e^{-tL}, \, t > 0 \}$ is  ${\mathcal R}$-bounded on $L^p(M)$. Then the same conclusions as before hold on $L^p(M)$. 
\end{theorem}
The sums over $k$ used here can be  taken up to  some $K \in \NN$, the  estimate is independent of $K$. 

\begin{proof} 
By a simple density argument we can assume that $f_k \in L^2(M) \cap L^p(M)$.\\
Let $f \in L^2(M) \cap L^p(M)$ and set $I(x) := \Big(  \int_0^\infty | \Gamma\,  F(tL) f (x) |^2 dt \Big)^{1/2}$ (if $\Gamma = \nabla$ then
actually, $I(x) := \Big( \int_0^\infty | \nabla\,  F(tL) f (x) |_x^2 dt \Big)^{1/2}$ but we ignore the subscript $x$ for the norm $| \cdot |$). 
By integration by parts,
\begin{eqnarray}
I^2 &=& \lim_{t\to \infty} t | \Gamma F(tL) f |^2 - 2 \int_0^\infty t  \Gamma\,  L F'(tL) f.  \Gamma\, F(tL)f \, dt \nonumber\\
&=& - 2 \int_0^\infty t  \Gamma\,  L F'(tL) f. \Gamma\, F(tL)f \, dt \nonumber\\
&\le& 2 \Big(  \int_0^\infty | \Gamma\, tL F'(tL) f |^2 dt \Big)^{1/2} I. \label{eq-I}
\end{eqnarray}
In order to justify the second equality we note that $\| \Gamma g \|_2 \le \| L^{1/2} g \|_2$ and hence  by the spectral resolution of $L$ 
\begin{eqnarray*}
\int_M t | \Gamma\, F(tL) f |^2 \, dx &=& \| \sqrt{t} \Gamma\, F(tL) f \|_2^2\\ \\
&\le&  \| \sqrt{t} L^{1/2} F(tL) f \|_2^2\\
&=& \int_0^\infty |H(t\lambda)|^2 \, dE_\lambda(f,f)
\end{eqnarray*}
where $|H(z)|^2 = |z| |F(z)|^2$. Since $F$ decays as $\frac{1}{|z|^\delta}$ at infinity with some $\delta > \frac{1}{2}$, it follows that $|H(z)|^2$ is bounded and  $|H(t \lambda)|^2 \to 0$ as $t \to \infty$ for all $\lambda \in (0, \infty)$. Note that if $\lambda = 0$, then $H(0) = 0$. We conclude by the dominated convergence theorem that
$ \int_0^\infty |H(t\lambda)|^2 \, dE_\lambda(f,f) \to 0$ as $t \to \infty$. After extraction of a subsequence if necessary we obtain \eqref{eq-I}. \\
Set   $G(z) := z F'(z)$.  It follows from  \eqref{eq-I} that
\begin{equation}\label{ineq-I}
\Big(  \int_0^\infty | \Gamma\,  F(tL) f |^2 dt \Big)^{1/2} \le 2 \Big(  \int_0^\infty | \Gamma\, G(tL) f |^2 dt \Big)^{1/2}.
\end{equation}
The  gain here is that the function $G$ on the RHS has decay at $0$ (and also at infinity) whereas  $F$ was not assumed to have such decay at $0$. This will allow us to use square function estimates.

In order to continue let $ {\mathcal H} := L^2((0, \infty), \frac{dt}{t})$\footnote{in the sequel, for a given $g \in {\mathcal H}$, we use the notation $\| g(t) \|_{\mathcal H}$ instead of  $\| g \|_{\mathcal H}$ or $\| g(.) \|_{\mathcal H}$. This makes reading easier since the variable $t$ appears at several places.} and set 
\[ J :=  \Big\|  \Big( \sum_k \int_0^\infty | \Gamma\,  m_k(tL) F(tL) f_k |^2 dt \Big)^{1/2} \Big\|_p.\]
We apply \eqref{ineq-I} to the function $z \mapsto F(z) m_k(z)$. Then $J$ is bounded (up to a constant) by $J_1 + J_2$ with
\[ J_1 := \Big\|  \Big( \sum_k \int_0^\infty | \Gamma\,  G(tL) m_k(tL) f_k |^2\, dt \Big)^{1/2} \Big\|_p\]
and
\[ J_2 := \Big\|  \Big( \sum_k \int_0^\infty | \Gamma\,  (tL) F(tL) m_k'(tL) f_k |^2\, dt \Big)^{1/2} \Big\|_p.\]
We first estimate $J_1$. By  the Kahane inequality
\begin{eqnarray*}
J_1^p & \le &  \Big\|  \Big( \sum_k \int_0^\infty | \Gamma\,  G(tL) m_k(tL) f_k |^2\, dt \Big)^{1/2} \Big\|_p^p\\
 &=&   \Big\|  \Big( \sum_k \| \sqrt{t}\, \Gamma\,  G(tL) m_k(tL) f_k \|_{\mathcal H}^2 \Big)^{1/2} \Big\|_p^p \\
&\approx& \Big\|  \Big(  \EE\, \| \sum_k {\gothic r}_k  \sqrt{t}\, \Gamma\,  G(tL)  m_k(tL) f_k \|_{\mathcal H}^p \Big)^{1/p} \Big\|_p^p\\
&=& \EE\,  \Big\|   \|  \sqrt{t}\,  \Gamma\,  (I + tL)^{-\delta'} (I + tL)^{\delta'} G(tL) \sum_k {\gothic r}_k  m_k(tL) f_k \|_{\mathcal H}  \Big\|_p^p
\end{eqnarray*} 
where $\delta' \in (\frac{1}{2}, \delta)$. If $p \in (1,2]$, then $\{\sqrt{t}\, \Gamma \, e^{-tL}, \ t > 0 \}$ is ${\mathcal R}$-bounded on $L^p(M)$ by Corollary \ref{cor-LPSR}. If $p \in (2, \infty)$ this ${\mathcal R}$-boundedness was assumed in the theorem. Hence by  Proposition \ref{r-bb-resol} and Lemma \ref{lem1-0} the very last term is bounded (up to a constant) by
$$  \EE\,  \Big\|   \|   (I + tL)^{\delta'} G(tL) \sum_k {\gothic r}_k  m_k(tL) f_k \|_{\mathcal H}  \Big\|_p^p.$$
Set $H(tL) := (I + tL)^{\delta'} G(tL)$
and use again the Kahane inequality to obtain
\begin{equation}\label{eq-II-1}
J_1 \le C\, \EE\,  \Big\|   \|   \sum_k {\gothic r}_k  m_k(tL) H(tL) f_k \|_{\mathcal H}  \Big\|_p.
\end{equation}
In order to continue we approximate the norm $\|   \sum_k {\gothic r}_k  m_k(tL) H(tL) f_k \|_{\mathcal H}$ by the discrete  sum
$$\Big( \sum_{j \in \ZZ} \big{|} \sum_k {\gothic r}_k m_k( 2^{j/N}L) H( 2^{j/N}L) f_k \big{|}^2  \Big)^{1/2},$$
when $N \to + \infty$. We  deal with the term
$$ \tilde{J}_{1,N} := \EE\,  \Big\|    \big( \sum_{j \in \ZZ} \big{|} \sum_k {\gothic r}_k m_k( 2^{j/N}L) H( 2^{j/N}L) f_k \big{|}^2  \big)^{1/2} \Big\|_p.$$
We use a double randomisation argument.  Take another independent Rademacher variables $\tilde{{\gothic r}_j}$ with expectation $\tilde{\EE}$ and apply the Kahane-Khintchine inequality to see that 
$$ \tilde{J}_{1,N} \approx \EE\, \tilde{\EE}\,  \Big\|   \sum_{j } \tilde{{\gothic r}_j} \sum_k {\gothic r}_k m_k( 2^{j/N}L) H( 2^{j/N}L) f_k  \Big\|_p.$$
Since $L^p$ has Pisier's contraction property,  the double expectation is equivalent to a single one with a doubly indexed Rademacher variables ${\tilde{\tilde{\gothic r}}}_{j,k}$ (with expectation $\tilde{\tilde{\EE}}$) as follows
$$ \tilde{\tilde{\EE}}\,  \Big\|   \sum_{j,k } {\tilde{\tilde{\gothic r}}}_{j,k} m_k( 2^{j/N}L) H( 2^{j/N}L) f_k  \Big\|_p.$$
See \cite[Propositions 7.5.3 and 7.5.4]{HytonenII}. 
If $\sup_{k} \| m_k \|_{H^\infty(\Sigma(w_p))} < \infty$, then $\{ m_k(L), \ k \ge 1 \}$ is ${\mathcal R}$-bounded on $L^p(M)$ by 
\cite{KW} or \cite[Theorem~10.3.4]{HytonenII} and the fact that $L$ has bounded holomorphic functional calculus on $L^p(M)$. Using this in the forgoing estimates, we arrive at
$$\tilde{J}_{1,N} \le C\, \sup_{k} \| m_k \|_{H^\infty(\Sigma(w_p))}\, \tilde{\tilde{\EE}}\,  \Big\|   \sum_{j,k } {\tilde{\tilde{\gothic r}}}_{j,k}  H( 2^{j/N}L) f_k  \Big\|_p.$$
We argue as before by using again Pisier's contraction property to see that the term on the RHS is equivalent to \begin{equation*}
\EE\, \tilde{\EE}\,  \Big\|   \sum_{j } \tilde{{\gothic r}_j} \sum_k {\gothic r}_k H( 2^{j/N}L) f_k  \Big\|_p 
\approx \EE\,  \Big\|   \big( \sum_{j }  \big{|} \sum_k {\gothic r}_k H( 2^{j/N}L) f_k \big{|}^2 \big)^{1/2} \Big\|_p.
\end{equation*}
We let $N \to \infty$ and we obtain the bound 
$$J_1 \leq  C\, \EE\, \Big\| \big( \int_0^\infty \big{|} H(tL) \sum_k {\gothic r}_k  f_k \big{|}^2\, \frac{dt}{t} \big)^{1/2} \Big\|_p.$$
Let $\omega' \in (\arcsin|\frac{2}{p}-1|, \omega_p)$. Using the fact that $F$ has decay $\frac{1}{|z|^{\delta}}$ at infinity, it follows  easily from the Cauchy formula that $F'(z)$ decays at least as $\frac{1}{|z|^{1+\delta}}$
for $z \in \Sigma(\omega'_p)$. This implies that the function $H(z) := (1+z)^{\delta'} G(z) = 
(1+z)^{\delta'} z F'(z)$ decays at least as $\frac{1}{|z|^{ \delta -\delta'}}$ at infinity.  On the other hand, since $| F'(z) | \le \frac{C}{|z|^{1-\epsilon_0}}$ as $z \to 0$ it follows that $H  \in H^\infty_0(\Sigma(\omega'_p))$. Therefore, we can use the square function estimate \eqref{eq-sqf} for $H(tL)$ on $L^p(M)$ and we obtain 
\begin{equation}\label{j1-00}
J_1 \le C\, \sup_{k} \| m_k \|_{H^\infty(\Sigma(w_p))}\,  \Big\| \Big( \sum_k |f_k|^2 \Big)^{1/2} \Big\|_p,
\end{equation} 
which gives the desired estimate for $J_1$. 

Next we estimate $J_2$. The proof is similar to the previous one and hence we will not repeat all the details. Let 
again $\delta' \in (\frac{1}{2}, \delta)$ and take $\epsilon \in (0, \epsilon_0)$ such that $\delta' + \epsilon  < \delta$. Set
$H(tL) := (I + tL)^{\delta'} (tL)^\epsilon F(tL)$.  The function $H$ has decay both at $0$ and infinity and hence $H(tL)$ satisfies a square function estimate.  Similarly to \eqref{eq-II-1} we have 
\begin{equation*}
J_2 \le C\, \EE\,  \Big\|   \|   \sum_k {\gothic r}_k   (tL)^{1-\epsilon} m_k'(tL) H(tL) f_k \|_{\mathcal H}  \Big\|_p.
\end{equation*}
We use the double randomisation argument  as before as well as the square function estimate for $H(tL)$ to obtain   \eqref{j1-00} for $J_2$ with  $z \mapsto z^{1-\epsilon} m_k'(z)$ at the place of $m_k$.  Since $\epsilon < \epsilon_0$ it is obvious that 
$| z^{1-\epsilon} m_k'(z) | \le M_1$ for $|z | \le 1$. For $|z| \ge 1$, the inequality, $| z^{1-\epsilon} m_k'(z)  | \le 
\| m_k \|_{H^\infty(\Sigma(w_p))}$ follows from the Cauchy formula. Thus, we obtain $J_2 \le C(M_0 + M_1)$. This  finishes the proof of the theorem. 
\end{proof}

\begin{remark} In the estimate of $J_2$, the role of $(tL)^\epsilon$ is to ensure  that the function $H(z) = (1+ z)^{\delta'} z^\epsilon F(z)$ has decay at zero which then allows  to use a square function estimate for $H(tL)$.  If the function $F$ itself has a decay at zero,
i.e., $F \in H_0^\infty(\Sigma(w_p))$,  then the term $(tL)^\epsilon$ is not needed. This means that  we can take $\epsilon = 0$ in the previous proof.  Next, the term $z m_k'(z)$ is bounded by  $\| m_k \|_{H^\infty(\Sigma(w_p))}$ by the Cauchy formula. Thus,   \eqref{eq2-1} can be replaced in this case by 
\begin{equation*}
\Big\|  \Big( \sum_k \int_0^\infty | \Gamma\,  m_k(tL) F(tL) f_k |^2 dt \Big)^{1/2} \Big\|_p \le C_F \| m_k \|_{H^\infty(\Sigma(w_p))}
  \Big\| \Big( \sum_k |f_k|^2 \Big)^{1/2} \Big\|_p
\end{equation*}
for all bounded holomorphic  functions $m_k$ on $\Sigma(w_p)$. 
\end{remark} 

 In the next result we aim  to consider functions $m_k$ which are not holomorphic.  We will take such functions in a Sobolev space $W^{\delta,2}$ on the half-line $(0, \infty)$. In order to so, we make some assumptions on the manifold $M$.\\
We assume that $M$ satisfies the volume doubling property
\begin{equation}\label{doubl}
v(x,2r) \leq C v(x,r),
\end{equation} 
where $v(x,r)$ denotes the volume of the ball of centre $x \in M$ and radius $r > 0$. The constant $C$ is independent of $x$ and $r$. Note that \eqref{doubl} implies the existence of $C, N > 0$ such that for all $x$ in $M$, $r>0$ and $\lambda \ge 1$ 
\begin{equation}\label{doubl2}
v(x,\lambda r) \leq C \lambda^N v(x,r).
\end{equation}
Next, we assume that the  heat kernel $p_t(x,y)$ of $\Delta$  satisfies the Gaussian upper bound
\begin{equation}\label{gauss}
p_t(x,y) \leq \frac{C}{v(x,t^{1/2})} e^{-c\frac{d^2(x,y)}{t}}
\end{equation} 
for some positive constants $c$ and $C$ and all $x, y \in M$ and $t > 0$.
Since $V$ is non-negative, it is a standard fact that the semigroup $e^{-tL}$ is pointwise dominated by $e^{-t\Delta}$ (see, e.g. \cite[Section 4.5]{Ouh05}) and in particular, 
the heat kernel $k_t(x,y)$ associated with $L = \Delta + V$ 
satisfies the same Gaussian upper bound. We have

\begin{theorem}\label{thm-mult} Suppose that $M$ satisfies \eqref{doubl} and \eqref{gauss}. Let $m_k: [0, \infty) \to \CC$ with support contained in $[\frac{1}{2}, 2]$ for every  $k$. Let $p \in (1, 2]$. Then there exist $C > 0$, independent of $m_k$, and $\delta > 0$ such that
\begin{eqnarray}\label{eq2-2}
&&\Big\|  \Big( \sum_k \int_0^\infty | \nabla  m_k(tL) f_k |^2 dt \Big)^{1/2} \Big\|_p 
+ \Big\|  \Big( \sum_k \int_0^\infty | \sqrt{V} m_k(tL) f_k |^2 dt \Big)^{1/2} \Big\|_p \nonumber\\
&&  \le C\,  \sup_{k} \| m_k \|_{W^{\delta,2}}\,  \Big\| \Big( \sum_k |f_k|^2 \Big)^{1/2} \Big\|_p
\end{eqnarray}
for all $f_k \in L^p(M)$.\\
 For a given $p \in (2, \infty)$, suppose in addition that  $\{ \sqrt{t} \nabla\, e^{-tL}, \ t > 0 \}$ is ${\mathcal R}$-bounded on $L^p(M)$.  Then 
 \begin{equation}\label{eq2-3}
 \Big\|  \Big( \sum_k \int_0^\infty | \nabla  m_k(tL) f_k |^2 dt \Big)^{1/2} \Big\|_p 
 \le C\,  \sup_{k} \| m_k \|_{W^{\delta,2}}\,  \Big\| \Big( \sum_k |f_k|^2 \Big)^{1/2} \Big\|_p
\end{equation}
for all $f_k \in L^p(M)$. If $\{ \sqrt{t}\, \sqrt{V}\, e^{-tL}, \ t > 0 \}$ is ${\mathcal R}$-bounded on $L^p(M)$,  then the same estimate holds with $\sqrt{V}$ in place of $\nabla$. 
\end{theorem} 
\begin{proof}  

Recall that by Corollary \ref{cor-LPSR}, the set $\{ \sqrt{t}\, \Gamma\, e^{-tL}, \ t > 0\}$ is 
${\mathcal R}$-bounded on $L^p(M)$ for all $p \in (1,2]$.
Define
\[ J := \Big\|  \Big( \sum_k \int_0^\infty | \Gamma\,   m_k(tL) f_k |^2 dt \Big)^{1/2} \Big\|_p.
\]
As in the  proof of Theorem \ref{thm-hol} we use the Kahane inequality to obtain
\begin{eqnarray*}
J^p &=& \Big\| \Big( \sum_k \| \sqrt{t}\, \Gamma\, m_k(tL)f_k \|_{\mathcal H}^2 \Big)^{1/2} \Big\|_p^p\\
&\approx&  \Big\| \Big( \EE\,  \| \sqrt{t}\, \Gamma\, \sum_k {\gothic r}_k  m_k(tL)f_k \|_{\mathcal H}^p \Big)^{1/p} \Big\|_p^p\\
&=& \EE\, \Big\|   \| \sqrt{t}\, \Gamma\, e^{-tL} \sum_k {\gothic r}_k  \varphi_k(tL)f_k \|_{\mathcal H} \Big\|_p^p
\end{eqnarray*}
where $\varphi_k(\lambda) = e^{\lambda} m_k(\lambda)$. Using the ${\mathcal R}$-boundedness of $\sqrt{t}\, \Gamma\, e^{-tL}$  and Lemma \ref{lem1-0} we obtain
\[
J^p \le C\, \EE \Big\|  \| \sum_k {\gothic r}_k  \varphi_k(tL) f_k \|_{\mathcal H} \Big\|_p^p.
\]
Hence (use Kahane again)
\begin{eqnarray}\label{eq2-5}
J  &\le& C\, \Big( \int_M \EE\, \| \sum_k {\gothic r}_k \varphi_k(tL) f_k \|_{\mathcal H}^p \Big)^{1/p} \nonumber\\
&\le & C'\,  \Big( \int_M  \Big(  \sum_k \| \varphi_k(tL) f_k \|_{\mathcal H}^2 \Big)^{p/2} \Big)^{1/p} \nonumber\\
&=& C'\, \Big\|  \Big(  \sum_k \| \varphi_k(tL) f_k \|_{\mathcal H}^2 \Big)^{1/2}  \Big\|_p.
\end{eqnarray}
Now, since $L$ satisfies the Gaussian upper bound \eqref{gauss} and $M$ satisfies the doubling condition \eqref{doubl2} (in which we take $N$ to be the smallest possible), then it is known that $L$ satisfies spectral multiplier theorems. In particular, since $\varphi_k$ has compact support, one has
$\varphi_k (L)$ is bounded on $L^p(M)$ provided $\varphi_k \in W^{\alpha, 2}$ for some $\alpha > N | \frac{1}{2} -\frac{1}{p} | + \frac{1}{2}$. See \cite{DOS} or  \cite{COSY}, Theorem~A, and the references therein. Finally, \cite[Theorem~3.1]{DK} asserts that the RHS term in \eqref{eq2-5} is bounded by (up to a constant)
\[
\sup_k \| \varphi_k \|_{W^{\delta, 2}}\,  \Big\| \Big( \sum_k |f_k|^2 \Big)^{1/2} \Big\|_p
\]
with $\delta = \alpha + 1$. Since the support of $m_k$ is contained in $[\frac{1}{2}, 2]$, the quantities  $\| \varphi_k \|_{W^{\delta, 2}}$
and $\| m_k \|_{W^{\delta, 2}}$ are equivalent. This proves \eqref{eq2-2}.\\
For $p > 2$ the proof is the same since we assume here that  $\{ \sqrt{t}\, \Gamma\, e^{-tL}, \ t > 0 \}$ is ${\mathcal R}$-bounded on
$L^p(M)$. 
\end{proof}

\begin{remark}\label{rem2} 
1- In the proof we have taken $\delta = \alpha + 1$  with $\alpha > N | \frac{1}{2} -\frac{1}{p} | + \frac{1}{2}$. The latter  value is the order required in the regularity of spectral  multipliers under the sole conditions \eqref{doubl2} and \eqref{gauss}. There are however many situations where one has  sharp spectral multiplier results and hence a smaller order $\alpha$. This is the case if for example $L$ satisfies the so-called restriction estimate or if the corresponding Schr\"odinger group $e^{itL}$ satisfies global Strichartz estimates. We refer to \cite{DOS} and \cite{COSY}.\\
2- We assumed in the previous theorem that the functions $m_k$ are compactly supported. For more general functions, we may use Corollary~3.3 from \cite{DK} and obtain the same results under the condition
\[ \sum_n \sup_k \| \lambda \mapsto \sqrt{2^n \lambda} m_k (2^n \lambda ) \phi_0(\lambda) \|_{W^{\delta,2}} < \infty
\]
for some auxiliary non trivial function  $\phi_0$ having compact support in $(0, \infty)$.\\
 3- The assumption of the theorem for $p > 2$ is valid if  the Riesz transform $\nabla L^{-1/2}$  is bounded on $L^p(M)$. This latter property may not be satisfied in some case even for $L = \Delta$, especially when $p > m$ where $m$ is the dimension of $M$ (see \cite{CoulhonDuong}). For $L= \Delta + V$ we may have boundedness of the corresponding Riesz transform (together with $\sqrt{V} L^{-1/2}$) on $L^p$ under some integrability conditions on $V$ (cf. \cite{AO} or \cite{Devyver}). In the Euclidean case $M = \RR^m$, $\nabla L^{-1/2}$ is bounded on $L^p$ for a range of $p > 2$ if $V$ is in an appropriate reverse H\"older class (cf. \cite{BenAli} or \cite{Shen}). We shall come back to these examples again in Section \ref{sec5} in which we will see that the Littlewood-Paley-Stein functional might be unbounded outside the range of $p$ for which we have boundedness of the Riesz transform. \\
 4- In \cite{Ouhabaz}, it is shown for a class of potentials $V$ that the boundedness on $L^p(\RR^m)$ for some $p > m$ of the Littlewood-Paley-Stein functional 
  \[ 
  H_\nabla(f) = \Big(  \int_0^\infty | \nabla  e^{-tL} f |^2 dt \Big)^{1/2}
 \]
 implies $V = 0$.
\end{remark}

\section{Other Littlewood-Paley-Stein functionals}\label{sec3}

Following \cite{Chen2014}, the local Riesz tranform for $L$ is defined by $R_{loc} := \nabla (L + I)^{-1/2}$ and the Riesz transform at infinity is 
$R_\infty := \nabla L^{-1/2} e^{-L}$. Then (cf. \cite{Chen2014} Theorem~1.5), the Riesz transform is bounded on $L^p(M)$ {\it if and only if}
$R_{loc}$ and $R_\infty$ are both bounded on $L^p(M)$.
The direct implication is obvious. For the converse, we write
\begin{align}\label{Chen}
\| \nabla L^{-1/2} f \|_p &\leq \|\nabla L^{-1/2} e^{-L} f {\|}_p + \| \nabla (L+I)^{-1/2} (L+I)^{1/2}L^{-1/2}  ( I - e^{-L}) f \|_p\nonumber \\
&\leq C\, \Big( \| f\|_p +\| (L+I)^{1/2} L^{-1/2}  ( I - e^{-L}) f \|_p \Big).
\end{align}
Since $(L+I)^{1/2}L^{-1/2}  ( I - e^{-L}) = \varphi(L)$  with $\varphi(z) = \sqrt{z+1} \frac{1- e^{-z}}{\sqrt{z}}$ we use the boundedness of the holomorphic functional calculus on $L^p$ and obtain $\|\nabla L^{-1/2} f \|_p \leq C\,  \| f\|_p$.\\
The same observation is also valid for $\sqrt{V}$ in place of $\nabla$. 

\medskip
We define the local vertical Littlewood-Paley-Stein functional and the vertical Littlewood-Paley-Stein functional at infinity associated with $L$ by 
\[
H^{loc}_\Gamma(f) := \Big(  \int_0^1 | \Gamma e^{-tL} f |^2 dt \Big)^{1/2} \ \ {\rm and} \ \ 
H^{(\infty)}_\Gamma(f) := \Big( \int_1^\infty | \Gamma e^{-tL} f |^2 dt \Big)^{1/2}. 
\]
We restrict our selves in this section to these Littlewood-Paley-Stein functionals but we can also deal with  general ones as in Theorem~\ref{thm-hol}  at least for functions $F$ which have some exponential decay at infinity. \\
As we have already remarked in the introduction,  these functionals are always bounded on $L^p(M)$ for $p \in (1,2]$. Thus,  we consider in the sequel the case $p > 2$, only. 

\begin{proposition}\label{thm-loc} 
Let $\Gamma$ be either $\nabla$ or multiplication by $\sqrt{V}$ and let $p \in (2, \infty)$. \\
1) If the set $\{ \sqrt{t}\, \Gamma\, e^{-tL}, \ t \in (0, 1] \}$ is ${\mathcal R}$-bounded on $L^p(M)$, 
 then the local vertical Littlewood-Paley-Stein functional $H^{loc}_\Gamma$ is  bounded on $L^p(M)$. \\
2) Similarly, if the set $\{ \sqrt{t-1}\, \Gamma\, e^{-tL}, \ t  > 1 \}$ is ${\mathcal R}$-bounded on $L^p(M)$, then $H^{(\infty)}_\Gamma$ is bounded on $L^p(M)$.  
\end{proposition}
\begin{proof}  The arguments  are similar to  the proof of Theorem \ref{thm-hol}. 
For assertion 1),  the same proof as \eqref{ineq-I} gives
 \begin{equation}\label{ineq-I-1}
 \Big(  \int_0^1 | \Gamma\,  e^{-tL} f |^2 dt \Big)^{1/2} \le 2 \Big( | \Gamma\, e^{-L}f | +  \Big(  \int_0^1 | \Gamma\, tL e^{-tL} f |^2 dt \Big)^{1/2}  \Big). 
 \end{equation}
 Note that the ${\mathcal R}$-boundedness assumption implies that $\Gamma\, e^{-L}$ is a bounded operator on $L^p(M)$. The second term on the RHS of \eqref{ineq-I-1} coincides (up to a constant) with
 \[
 \Big(  \int_0^1 | \sqrt{\frac{t}{2}}\, \Gamma\, e^{-\frac{t}{2} L}  \sqrt{\frac{t}{2}} L e^{-\frac{t}{2}L} f |^2 dt \Big)^{1/2}. 
 \]
Since $\{ \sqrt{\frac{t}{2}}\, \Gamma\, e^{-\frac{t}{2} L}, \ t \in (0, 1] \}$ is ${\mathcal R}$-bounded we apply  Lemma \ref{lem1-0}.  Note that the term $  \sqrt{\frac{t}{2}} L e^{-\frac{t}{2}L} f$ is in $L^2((0,\infty), dt)$ by a square function estimate.
 
 In order to prove assertion 2),  we start  by writing
 $$\int_1^\infty | \Gamma\,  e^{-tL} f |^2 dt  = \left[ (t-2) | \Gamma\,  e^{-tL} f |^2 \right]_1^\infty + 2 \int_1^\infty (t-2) \Gamma L e^{-tL} f.\Gamma e^{-tL}f dt$$
 and proceed exactly as in the proof of   \eqref{ineq-I} to obtain 
 \begin{equation}\label{ineq-I-2}
 \Big(\int_1^\infty | \Gamma\,  e^{-tL} f |^2 dt \Big)^{1/2} \le 2 \Big( | \Gamma\, e^{-L} f | +  \Big( \int_1^\infty | (t-2) \Gamma\, L e^{-tL} f |^2 dt \Big)^{1/2} \Big). 
 \end{equation}
Next, since $\Gamma e^{-tL}$ has $L^p$-norm bounded by $\frac{C}{\sqrt{t}}$, the part  $\Big( \int_1^2 | (t-2) \Gamma\, L e^{-tL} f |^2 dt \Big)^{1/2}$ is obviously bounded on $L^p(M)$. It remain to deal with the part involving $t \ge 2$. This part coincides with  (up to constant)
\[
\Big( \int_2^\infty | \sqrt{\frac{t}{2} -1}\, \Gamma\, e^{-\frac{t}{2}L}\, \sqrt{\frac{t}{2} -1}\,  L e^{-\frac{t}{2}L} f |^2 dt \Big)^{1/2}.
\]
Now, we use the ${\mathcal R}$-boundedness of  $\{ \sqrt{\frac{t}{2} -1}\, \Gamma\, e^{-\frac{t}{2}L}, \ t >2 \}$, Lemma \ref{lem1-0} and a square function estimate for the term $\sqrt{\frac{t}{2} -1}\, L e^{-\frac{t}{2}L} f$ to obtain $2)$.
\end{proof}

We have the following version of  Proposition \ref{riesz-rbb}. 
\begin{proposition}\label{rbb-loc} 
Let $p \in (1,\infty)$. 
If the local Riesz transform  $\Gamma (L+I)^{-1/2}$ is bounded on $L^p(M)$, then  $\{ \sqrt{t}\, \Gamma\, e^{-tL}, \ t \in (0, 1] \}$ is ${\mathcal R}$-bounded on $L^p(M)$. Similarly, if the Riesz transform at infinity $\Gamma L^{-1/2}e^{-L}$ is bounded on $L^p(M)$, then 
$\{ \sqrt{t-1}\, \Gamma\, e^{-tL}, \ t > 1 \}$ is ${\mathcal R}$-bounded on $L^p(M)$.
\end{proposition}
\begin{proof} The proof of the first assertion is exactly the same as for Proposition \ref{riesz-rbb}. We prove the second one.  Let $f_k \in L^p(M)$ and $t_k > 1$ for $k=1,...,n$.  We have 
\begin{eqnarray*}
\EE \| \sum_{k=1}^n {\gothic r}_k \sqrt{t_k-1}\, \Gamma\, e^{-t_k L} f_k \|_p 
&=& \EE  \| \Gamma\,  L^{-1/2} e^{-L}  \sum_{k=1}^n {\gothic r}_k ((t_k-1) L)^{1/2} e^{-(t_k-1) L} f_k \|_p\\
&\le& C\, \EE  \|  \sum_{k=1}^n {\gothic r}_k ((t_k-1) L)^{1/2} e^{-(t_k-1) L} f_k \|_p.
\end{eqnarray*}
We finish the proof by appealing again to the ${\mathcal R}$-boundedness of the holomorphic functional calculus.
\end{proof}

It is an interesting question whether the boundedness of the Littlewood-Paley-Stein functional implies the boundedness of the Riesz transform. For $L= \Delta$ on $\RR^m$ this is true and very easy to prove (see \cite{Stein}, p. 52-54). Note however that this uses heavily the fact that $\nabla $ and $\Delta$ commute, a fact which is rarely satisfied outside the Euclidean context. If $L = \Delta$ and $M$ satisfies \eqref{doubl2} and $L^2$-Poincar\'e inequalities, then the $L^p$-boundedness of $H$ implies boundedness of the Riesz transform
on $L^r$ for $r \in (1, p)$. Indeed, the boundedness of $H$ implies that
$\| \nabla e^{-t \Delta} \|_p \le \frac{C}{\sqrt{t}}$ by Proposition \ref{prop-multiplication} below. The latter inequality implies the boundedness of the Riesz transform on $L^r(M)$ for $r < p$, see \cite{ACDH} or \cite{BF}. \\
In general, we do not have an answer to the previous question but we make some observations below.  Let $d$ be the exterior derivative on differential forms and let $d^*$ its formal adjoint. One defines the 
Hodge-de Rham Laplacian $\vec{\Delta}$ on $1$-differential forms by the formal expression  $\vec{\Delta} = dd^* + d^*d$. It is the self-adjoint operator associated with the symmetric  bilinear form
$$\vec{\mathfrak{a}}(w,\eta) = \int_M d w.d\eta + \int_M d^*w.d^*\eta,$$
for $w$ and $\eta$ in the Sobolev space of $1$-forms such that $|w|^2, | d w |^2 + |d^* w|^2$ are integrable on $M$. As a self-adjoint, $-\vec{\Delta}$ generates a $C_0$-semigroup $(e^{-t \vec{\Delta}})$ on the $L^2$-space of 
$1$-forms. 
We also recall  the commutation property $d \vec{\Delta} = \Delta d$. The reader can consult \cite{Driver} and the references there. 

Let $p \in (1, \infty)$ and suppose that $\vec{\Delta}$ satisfies the (weak) lower square function estimate
\begin{equation}\label{eq3-1}
\| e^{-\vec{\Delta}} w \|_p \le C\,  \Big\| \Big( \int_1^\infty  | {\vec{\Delta}}^{1/2} e^{-t\vec{\Delta}} w |^2 dt\Big)^{1/2} \Big\|_p.
\end{equation}
Then the boundedness on $L^p(M)$ of the Littlewood-Paley-Stein functional at infinity implies the boundedness on $L^p(M)$ of Riesz transform at infinity (compare with \cite[Theorem 5.1]{CoulhonCPAM}). Indeed, we chose $w = d \Delta^{-1/2} f$ for  $ f$  in the range of $\Delta^{1/2}$, and notice that 
$e^{-\vec{\Delta}} d \Delta^{-1/2} f = d \Delta^{-1/2} e^{-\Delta} f$ 
and $\vec{\Delta}^{1/2} e^{-t\vec{\Delta}} d \Delta^{-1/2} f = d e^{-t\Delta} f$. Then \eqref{eq3-1} gives
$$
 \| R_\infty f \|_p \le C\,  \Big\| \Big( \int_1^\infty  | d e^{-t \Delta} f |^2 dt \Big)^{1/2} \Big\|_p \le C'\, \| f \|_p.$$
If for example the Ricci curvature is bounded from below, then the local Riesz transform is bounded on $L^p(M)$ for all 
$p \in (1, \infty)$ (cf. \cite{Bakry}). This together with the observation \eqref{Chen}  imply  the boundedness of the Riesz transform on $L^p(M)$. 

The next observation is that if we have the following Littlewood-Paley-Stein estimate for $\vec{\Delta}$ on $1$-forms
\begin{equation}\label{eq3-2}
 \Big\| \Big( \int_0^\infty  | d^* e^{-t \vec{\Delta}} w |^2 dt\Big)^{1/2} \Big\|_p \le C\, \| w \|_p, 
 \end{equation}
then the Riesz transform $d \Delta^{-1/2}$ is bounded on $L^q(M)$, where $\frac{1}{p} + \frac{1}{q} = 1$. Indeed, using the lower square function estimate for $\Delta$ and the commutation property we obtain 
\begin{eqnarray*} 
\| d^* w \|_p &\le& C\, \Big\| \Big( \int_0^\infty  | \Delta^{1/2} e^{-t \Delta}  d^* w |^2 dt\Big)^{1/2} \Big\|_p\\
&=& C\, \Big\| \Big( \int_0^\infty  | d^* e^{-t \vec{\Delta}}  \vec{\Delta}^{1/2} w |^2 dt\Big)^{1/2} \Big\|_p\\
&\le& C'\, \| \vec{\Delta}^{1/2} w \|_p.
\end{eqnarray*}
This means that the Riesz transform $d^* \vec{\Delta}^{-1/2}$ is bounded on $L^p(\Lambda^1T^*M)$ into $L^p(M)$. The adjoint is then bounded on $L^q(M)$. But the adjoint is exactly the Riesz transform $d \Delta^{-1/2}$ (use the commutation property again). 

We also mention the following related result.  It is taken from \cite{CoulhonCPAM} for $L= \Delta$ and \cite{Ouhabaz} for $L = \Delta + V$. We reproduce the proof for the sake of completeness.  

\begin{proposition}\label{prop-multiplication}
Let $p \in (1, \infty)$ and set $\Gamma = \nabla$ or $\sqrt{V}$. Suppose that 
\begin{equation}\label{H-L}
  \Big\| \Big( \int_0^\infty  | \Gamma e^{-t L} f |^2 dt\Big)^{1/2} \Big\|_p \le C\, \| f \|_p
\end{equation}
for all $f \in L^p(M)$. Then 
\begin{equation}\label{32}
\| \Gamma f \|_p \le C'\, \| L f \|_p^{1/2} \| f \|_p^{1/2}
\end{equation}
for $f$ in the domain of $L$, seen as an operator on $L^p(M)$. 
\end{proposition}
 
 \begin{proof} Set $P_t := e^{-t \sqrt{L}}$ the Poisson semigroup associated with $L$ and fix $f \in L^2(M) \cap D(L)$. 
By integration by parts,
$$\| \nabla P_t f \|_2^2  + \| \sqrt{V} P_t f \|_2^2 = \| L^{1/2} P_t f \|_2^2.$$
In particular, 
$$\| \Gamma P_t f \|_2 \le \frac{C}{t} \|f\|_2 \to 0 \ \text{as} \ t \to + \infty.$$
The same arguments show that $t \| \Gamma L^{1/2} P_t f \|_2 \to 0$ as $t \to +\infty$. Therefore,
\begin{eqnarray*}
| \Gamma f  |^2 &=& - \int_0^\infty \frac{d}{dt} | \Gamma P_t f |^2 dt\\
&=& - \left[ t \frac{d}{dt} | \Gamma P_t f |^2 \right]_0^\infty  + \int_0^\infty \frac{d^2}{dt^2} | \Gamma P_t f |^2\,  t\, dt \\
&\le& \int_0^\infty \frac{d^2}{dt^2} | \Gamma P_t f |^2\, t\, dt \\
&=& 2 \int_0^\infty ( |\Gamma L^{1/2} P_t f|^2 + \Gamma L P_t f. \Gamma P_tf) t\, dt =: 2 (I_1 + I_2).
\end{eqnarray*}
On the other hand, \eqref{H-L} implies by the subordination formula for the Poisson semigroup $e^{-t\sqrt{L}}$ that
\[
G (f) :=  \Big( \int_0^\infty  | \Gamma e^{-t\sqrt{L}} f(x) |^2\, t\, dt \Big)^{1/2} 
\]
is also bounded on $L^p(M)$.
Observe that $\sqrt{I_1} = G(L^{1/2} f)$ and by the Cauchy-Schwartz inequality
 \begin{eqnarray*}
 | I_2| &\le& \Big( \int_0^\infty ( |\Gamma L P_t f|^2\,  t\, dt \Big)^{1/2} \Big( \int_0^\infty ( |\Gamma  P_t f|^2\, t\, dt \Big)^{1/2} \\
 &\le& G (L f) G (f).
 \end{eqnarray*}
 Hence for any $\epsilon > 0$
 \[ |\Gamma f | \le \sqrt{2} ( G(L^{1/2}f) + \epsilon G(f) + \frac{1}{\epsilon} G(Lf)).
 \]
 Taking the $L^p$-norm yields
 \[ \| \Gamma f \|_p \le C\, ( \| L^{1/2} f \|_p + \epsilon \| f \|_p + \frac{1}{\epsilon} \| Lf \|_p).
 \]
 We chose $\epsilon = \frac{\sqrt{\| L f\|_p}}{\sqrt{\| f\|_p}}$ and we obtain 
 \[ \| \Gamma f \|_p \le C\, \Big( \| L^{1/2} f \|_p +  \| f \|_p^{1/2}\| Lf \|_p^{1/2} \Big).
 \]
 It is well known that $\| L^{1/2} f \|_p$ is bounded (up to a constant) by $\| f \|_p^{1/2}\| Lf \|_p^{1/2}$ (see, e.g., \cite{Komatsu}, Proposition~5.5). Hence 
\eqref{32} is proved for $f \in D(L) \cap L^2(M)$. In order to extend this for all
 $f \in D(L)$ we take a sequence $f_n \in L^2(M) \cap L^p(M)$ which converges in the $L^p$-norm to $f$. We apply \eqref{32} to 
 $e^{-tL}f_n$ (for $t > 0$) and then  let $n \to + \infty$ and $t \to 0$. 
 \end{proof}

The standard argument of Stein which allows to prove  that the functional
\[ 
H(f) = \Big(  \int_0^\infty | \nabla e^{-tL} f |^2 dt \Big)^{1/2} + \Big(  \int_0^\infty | \sqrt{V} e^{-tL} f |^2 dt \Big)^{1/2}
\]
is always bounded on $L^p(M)$ for $p \in (1, 2]$ can be used to prove the following  proposition \footnote{we owe this observation to Sylvie Monniaux.}
\begin{proposition}\label{prop-stein} 
Let $p \in (1,2]$.  Then
\begin{equation}\label{stein}
 \int_0^\infty \| \nabla e^{-tL} f \|_p^2\, dt + \int_0^\infty \| \sqrt{V} e^{-tL} f \|_p^2\, dt \leq C \| f\|_p^2 
 \end{equation}
 for all $f \in L^p(M)$. For $q \in [2, \infty)$ we have 
 \begin{equation}\label{HHRq}
C\, \| f \|_q^2 \le \int_0^\infty \| \nabla e^{-tL} f \|_q^2\, dt + \int_0^\infty \| \sqrt{V} e^{-tL} f \|_q^2\, dt
 \end{equation}
 for all $f \in L^q(M) \cap L^2(M)$.
 
 \end{proposition}
\begin{proof}
It is enough to consider non-negative (and non-trivial) $f \in L^1(M) \cap L^2(M)$. Hence by irreducibility,  $e^{-t\Delta} f > 0$ (a.e. on $M$). We have 
\begin{align*}
\| \nabla e^{-tL}f\|_p^p  
&=  \int_M |\nabla e^{-t L}f |^p   (e^{-t L}f)^{\frac{p(p-2)}{2}}  (e^{-tL}f)^{\frac{p(2-p)}{2}}   dx \\
&\le \Big( \int_M |\nabla e^{-t L} f|^2 (e^{-tL} f)^{p-2} dx \Big)^{\frac{p}{2}} \Big( \int_M  (e^{-t L} f)^p dx \Big)^{\frac{2-p}{2}}\\
&\le \Big( \int_M |\nabla e^{-t L} f|^2 (e^{-tL} f)^{p-2} dx \Big)^{\frac{p}{2}} \| f \|_p^{ \frac{p(2-p)}{2}} .
\end{align*}
The same inequality holds when $\nabla$ is replaced by $\sqrt{V}$. Hence
\begin{align*}
\| \nabla e^{-tL}f\|_p^2 + \| \sqrt{V} e^{-tL}f\|_p^2 &\leq \Big( \int_M \left[ |\nabla e^{-tL} f|^2 + |\sqrt{V} e^{-tL} f|^2 \right] (e^{-t L} f)^{p-2}  dx \Big)  
\| f \|_p^{2-p} \\
&\le   C\,  \Big( \int_M - \frac{\partial}{\partial t} (e^{-tL }f)^p dx \Big) \| f\|_p^{2-p}.
\end{align*}
We integrate over $t \in [0, \tau]$ to obtain
\begin{align*}
\int_0^\tau \left[ \| \nabla e^{-tL}f\|_p^2 + \| \sqrt{V} e^{-tL}f\|_p^2 \right] \, dt 
&\leq C\,  \Big( \int_M \int_0^\tau - \frac{\partial}{\partial t} (e^{-tL }f)^p dx \Big) \| f\|_p^{2-p}\\
&= C\, \Big(  \| f \|_p^p -  \| e^{-\tau L} f \|_p^p \Big)  \| f\|_p^{2-p}\\
&\le C\, \| f\|_p^2.
\end{align*}
Letting $\tau \to \infty$ gives the desired result.\\
The proof of the lower estimate \eqref{HHRq} is postponed to the next section (see \eqref{HHRq-2}).  
\end{proof}

We have formulated the previous proposition for  Schr\"odinger operators on manifolds but it is also true for elliptic operators with non-smooth coefficients on domains. 

\section{Lower bounds}\label{sec4}
In this section we prove reverse inequalities for the Littlewood-Paley-Stein functionals. The strategy is classical and it is based on a duality argument which goes  back at least to  \cite[p. 85]{Stein2}.  

Let $L = \Delta + V$ be again a Schr\"odinger operator with a non-negative potential $V$. 
We shall need the assumption that $0$ in not an eigenvalue of $L$ as an operator on $L^2(M)$. Otherwise, if $M$ is compact without boundary, then $\nabla e^{-t\Delta} 1 = \nabla 1 = 0$, and hence a  lower estimate
cannot hold for $H_\nabla$. 

Observe that if $L f = 0$, then taking the scalar product with $f$ yields  $\nabla f = 0$ and $\sqrt{V} f = 0$. Hence, $f$ is constant. Therefore, $0$ cannot be an eigenvalue of $L$ if $M$ has infinite volume or if $V$ is not identically zero. 

The main result of this section reads as follows. 
\begin{theorem}\label{thm-reverse} Suppose that $0$ in not an eigenvalue of $L$ as an operator on $L^2(M)$. 
Let $m_k : [0, \infty) \to \CC$  in $L^2(0, \infty) \cap L^\infty(0,\infty)$ and such that
\begin{equation}\label{eq4-0}
\inf_k \| m_k \|_2^2 > 0.
\end{equation}
Let $p \in (1, \infty)$ and $q$ its conjugate number.  Suppose that there exists a constant $C > 0$ such that 
\begin{eqnarray}\label{eq4-1}
&&\Big\|  \Big( \sum_k \int_0^\infty | \nabla  m_k(tL) f_k |^2 dt \Big)^{1/2} \Big\|_p 
+ \Big\|  \Big( \sum_k \int_0^\infty | \sqrt{V} m_k(tL) f_k |^2 dt \Big)^{1/2} \Big\|_p \nonumber\\
&&  \le C\,   \Big\| \Big( \sum_k |f_k|^2 \Big)^{1/2} \Big\|_p
\end{eqnarray}
for all $f_k \in L^p(M)$.
Then there exists $C' > 0$ such that 
\begin{eqnarray}\label{eq4-2}
&& C'\, \Big\| \Big( \sum_k |g_k|^2 \Big)^{1/2} \Big\|_q\\
&&\le \Big\|  \Big( \sum_k \int_0^\infty | \nabla  m_k(tL) g_k |^2 dt \Big)^{1/2} \Big\|_q \nonumber
+ \Big\|  \Big( \sum_k \int_0^\infty | \sqrt{V} m_k(tL) g_k |^2 dt \Big)^{1/2} \Big\|_q
\end{eqnarray}
for all $g_k \in L^q(M) \cap L^2(M)$.
\end{theorem}

\begin{proof} We may assume without loss of generality that $k$ runs over $\{1,...,n\}$ for some finite $n$ (the constants $C$ and $C'$ are then independent of $n$). Let  $f_k \in L^p(M) \cap L^2(M)$ and $g_k \in L^q(M) \cap L^2(M)$. Set $F = (f_1,...,f_n)$  and $g = (g_1,...,g_n)$. We denote by $\langle ., . \rangle$ the usual scalar product in $\CC^n$. Then we have
\begin{eqnarray*}
&& \int_0^\infty \int_M \langle \nabla (m_1(tL)f_1,..., m_n(tL)f_n), \nabla (m_1(tL)g_1,..., m_n(tL)g_n) \rangle \, dx\, dt \\
&& + \int_0^\infty \int_M \langle \sqrt{V} (m_1(tL)f_1,..., m_n(tL)f_n), \sqrt{V} (m_1(tL)g_1,..., m_n(tL)g_n ) \rangle \, dx\, dt\\
&&= \int_0^\infty \int_M \langle  ( L m_1(tL)f_1,..., L m_n(tL)f_n), (m_1(tL)g_1,..., m_n(tL)g_n ) \rangle \, dx\, dt\\
&& = \int_0^\infty \int_M \langle (L |m_1|^2(tL)f_1,..., L |m_n|^2(tL)f_n), (g_1,..., g_n ) \rangle \, dx\, dt.
\end{eqnarray*}
The first equality is obtained by integration by parts (with respect to $x \in M$) in each coordinate and the second one uses the duality and the basic fact that the adjoint of $m_k(tL)$ is $\overline{m_k}(tL)$. For each $k$, set
\[
M_k (\lambda) := \int_\lambda^\infty | m_k(s) |^2 ds. 
\]
Then $M_k(\lambda) \to 0$ as $\lambda \to \infty$ and hence 
$M_k (tL) f \to 0$ in $L^2(M)$ as $t \to \infty$  for all $f \in L^2(M)$. In order to see this, we write by the spectral resolution
\[ \| M_k(tL) f \|_2^2 = ( |M_k|^2(tL)f, f) = \int_0^\infty |M_k(t\lambda)|^2 dE_\lambda(f,f).
\]
Since $0$ is not an eigenvalue of $L$, the latter integral is taken over $(0, \infty)$. Now, the positive measure $dE_\lambda(f,f)$ is finite, $|M_k(.)|^2$ is bounded on $(0, \infty)$ (remember $m_k \in L^2(0,\infty)$), and $M_k(t\lambda) \to 0$ as $t \to \infty$ for all $\lambda \in (0,\infty)$, then the result follows by the dominated convergence theorem.\\
Using again the spectral resolution we see that $\frac{d}{dt} M_k(tL) = -L |m_k|^2 (tL)$. From this we obtain 
\begin{eqnarray*}
&& \int_0^\infty \int_M \langle (L |m_1|^2(tL)f_1,..., L |m_n|^2(tL)f_n), (g_1,..., g_n ) \rangle \, dx\, dt\\
&&= \int_M \int_0^\infty  \langle - \frac{d}{dt} (M_1(tL)f_1,..., M_n(tL)f_n), (g_1,..., g_n ) \rangle \,dt\,  dx\\
&&= \int_M \langle ( f_1,..., f_n), (M_1(0)g_1,...,M_n(0)g_n ) \rangle dx.
\end{eqnarray*}
Using  the forgoing equalities, the Cauchy-Schwarz inequality (for $t$)  and H\"older's inequality (in $L^r(\CC^n)$) yields
\begin{eqnarray*}
&&\hspace{-.6cm} \left| \int_M \langle ( f_1,..., f_n), (M_1(0)g_1,...,M_n(0)g_n ) \rangle dx \right|\\
&&\hspace{-.6cm} \le \int_M \Big( \int_0^\infty | \nabla (m_1(tL)f_1,...m_n(tL)f_n) |^2 dt\Big)^{1/2}  \Big( \int_0^\infty | \nabla (m_1(tL)g_1,...m_n(tL)g_n) |^2 dt\Big)^{1/2} + \\
&&\hspace{-.6cm} \int_M \Big( \int_0^\infty | \sqrt{V} (m_1(tL)f_1,...m_n(tL)f_n) |^2 dt\Big)^{1/2}  \Big( \int_0^\infty | \sqrt{V} (m_1(tL)g_1,...m_n(tL)g_n) |^2 dt\Big)^{1/2}\\
&&\hspace{-.6cm}\le \left[  \Big\|  \Big( \sum_k \int_0^\infty | \nabla  m_k(tL) f_k |^2 dt \Big)^{1/2} \Big\|_p 
+ \Big\|  \Big( \sum_k \int_0^\infty | \sqrt{V} m_k(tL) f_k |^2 dt \Big)^{1/2} \Big\|_p \right] \times \\
&& \hspace{1cm} \left[  \Big\|  \Big( \sum_k \int_0^\infty | \nabla  m_k(tL) g_k |^2 dt \Big)^{1/2} \Big\|_q 
+ \Big\|  \Big( \sum_k \int_0^\infty | \sqrt{V} m_k(tL) g_k |^2 dt \Big)^{1/2} \Big\|_q \right]\\
&&\hspace{-.5cm} \le  C\,   \Big\| \Big( \sum_k |f_k|^2 \Big)^{1/2} \Big\|_p \times\\
&& \hspace{1cm} \left[  \Big\|  \Big( \sum_k \int_0^\infty | \nabla  m_k(tL) g_k |^2 dt \Big)^{1/2} \Big\|_q 
+ \Big\|  \Big( \sum_k \int_0^\infty | \sqrt{V} m_k(tL) g_k |^2 dt \Big)^{1/2} \Big\|_q \right]. 
\end{eqnarray*}
Hence, for 
$$J := \left[  \Big\|  \Big( \sum_k \int_0^\infty | \nabla  m_k(tL) g_k |^2 dt \Big)^{1/2} \Big\|_q 
+ \Big\|  \Big( \sum_k \int_0^\infty | \sqrt{V} m_k(tL) g_k |^2 dt \Big)^{1/2} \Big\|_q \right],$$
 we have proved
\[
| \int_M \langle F , (M_1(0)g_1,...,M_n(0)g_n ) \rangle dx | \le C\, \| F \|_{L^p(M, \CC^n)} J.
\]
This implies 
\[ 
\| (M_1(0)g_1,...,M_n(0)g_n ) \|_{L^q(M, \CC^n)} \le C\, J.
\]
Finally, we use \eqref{eq4-0} to finish the proof.
\end{proof} 

A particular case of  the above theorem shows that if 
\begin{equation}\label{HH}
 H(f) :=  \Big(  \int_0^\infty | \nabla e^{-tL} f |^2 dt \Big)^{1/2} + \Big(  \int_0^\infty | \sqrt{V} e^{-tL} f |^2 dt \Big)^{1/2}
\end{equation}
is bounded on $L^p(M)$, then there exists a constant $C > 0$ such that 
\begin{equation}\label{HHR}
C\, \|f \|_q \le \| H(f) \|_q
 \end{equation}
 for  $f \in L^q(M) \cap L^2(M)$. As we already mentioned in the introduction, $H$ is bounded on $L^p(M)$ for $p \in (1,2]$. Therefore, the  lower bound  \eqref{HHR} holds for all $q \in [2, \infty)$. This together with the triangle inequality for the $L^{\frac{q}{2}}$-norm 
 implies 
\begin{equation}\label{HHRq-2}
C\, \| f \|_q^2 \le  \int_0^\infty \| \nabla e^{-tL} f \|_q^2\, dt + \int_0^\infty \| \sqrt{V} e^{-tL} f \|_q^2\, dt.
\end{equation}
  This is the lower bound stated in Proposition \ref{prop-stein}.

 \medskip
 Recall from Section \ref{sec3} the local Littlewood-Paley-Stein functional 
 \[
 H^{loc}(f) := \Big(  \int_0^1 | \nabla e^{-tL} f |^2 dt \Big)^{1/2} +  \Big(  \int_0^1 | \sqrt{V} e^{-tL} f |^2 dt \Big)^{1/2}.
 \]
We have seen that the boundedness on $L^p(M)$ of the local Riesz transforms $\nabla (L+I)^{-1/2}$ and $\sqrt{V} (L+I)^{-1/2}$ imply the boundedness on $L^p(M)$ of $H^{loc}$. This together with the standard fact that the semigroup $e^{-tL}$ acts (as a contraction) on $L^p(M)$ imply that the functional
\begin{equation}\label{def-Q}
Q(f) := | e^{-L} f | + H^{loc}(f)
\end{equation} 
is also bounded on $L^p(M)$. The next proposition shows that a lower bound is also true for $Q$. More precisely,
\begin{proposition}\label{prop-Q} Let $p \in (1,\infty)$ and suppose that $H^{loc}$ is bounded on $L^p(M)$. Then there exists a constant $C > 0$ such that 
\[
C \| g \|_q \le \| Q(g) \|_q
\]
for all $g \in L^q(M) \cap L^2(M)$, $ \frac{1}{p} + \frac{1}{q} = 1$.
\end{proposition}
 \begin{proof} Let $f \in L^p(M) \cap L^2(M)$ and $g \in L^q(M) \cap L^2(M)$. We have
 \begin{eqnarray*}
 \int_0^1 \int_M \nabla e^{-tL}f.\nabla e^{-tL}g + \sqrt{V} e^{-tL}f.\sqrt{V} e^{-tL}g \, dx\, dt &=&
 \int_0^1 \int_M (L e^{-2tL}f ) g \, dx\, dt\\
 &=& - \frac{1}{2} \int_M \int_0^1 \frac{d}{dt}( e^{-2tL}f) g \, dx\, dt\\
 &=& \frac{1}{2} \int_M f  g \, dx - \frac{1}{2} \int_M (e^{-L} f) (e^{-L}g)\, dx.
 \end{eqnarray*}
 Therefore,
 \begin{eqnarray*}
  \left|  \int_M f  g \, dx \right| &\le& \int_M |e^{-L} f| |e^{-L}g|\, dx  + 2 \int_M H^{loc}(f) H^{loc} (g) \, dx\\
  &\le& 2 \int_M ( |e^{-L} f | + H^{loc} (f) ) ( | e^{-L} g |+  H^{loc} (g) )\\
  &\le&  2 \| Q(f) \|_p \| Q(g) \|_q\\
  &\le& C\, \| f \|_p \| Q(g) \|_q.
  \end{eqnarray*}
  The  latter inequality extends by density to all $f \in L^p(M)$ and the proposition follows.
 \end{proof}
 The final observation in this section is that if the Littlewood-Paley-Stein functional at infinity
 \[ 
 H^{(\infty)}(f) := \Big(  \int_1^\infty | \nabla e^{-tL} f |^2 dt \Big)^{1/2} +  \Big(  \int_1^\infty | \sqrt{V} e^{-tL} f |^2 dt \Big)^{1/2}
 \]
 is bounded on $L^p(M)$ for some $p \in (1,\infty)$, then 
 \[
 C\, \| e^{-2L} g \|_q \le \| H^{(\infty)} g \|_q
 \]
 for $ g \in L^q(M) \cap L^2(M)$, $ \frac{1}{p} + \frac{1}{q} = 1$. The proof is very similar to the previous one. When we integrate over $t$ on $[1, \infty)$ we obtain
 \begin{eqnarray*}
 \int_1^\infty \int_M \nabla e^{-tL}f.\nabla e^{-tL}g + \sqrt{V} e^{-tL}f.\sqrt{V} e^{-tL}g \, dx\, dt &=&
 - \frac{1}{2} \int_M \int_1^\infty \frac{d}{dt}( e^{-2tL}g) f \, dx\, dt\\
 &=& \frac{1}{2} \int_M f  e^{-2L}g \, dx
 \end{eqnarray*}
 and we proceed as before.

\section{Examples and counter-examples}\label{sec5}

In this section we discuss several examples.  We also  give a short review of some known results on the Riesz transform. The boundedness of the Riesz transform implies the boundedness of the Littlewood-Paley-Stein functionals.  We shall see that the examples for which the Riesz transform is not bounded are also examples for which the Littlewood-Paley-Stein functionals are unbounded.  

\medskip
\noindent \underline{\it The Laplacian.} \\
We start with the case $L = \Delta$ the (positive) Laplace-Beltrami operator on a manifold $M$.  We give examples of manifolds for which the Riesz transform $\nabla \Delta^{-1/2}$ is bounded on $L^p(M)$ (with values in the $L^p$-space  of vector fields). This subject has been studied for many years and  it is impossible to provide comprehensive bibliography here. Recall that if the Riesz transform is bounded then the Littlewood-Paley-Stein estimates of Section \ref{sec2} are satisfied on $L^p(M)$. The lower bounds of  Section \ref{sec4} are then satisfied. 

\medskip
- {\it Manifold with non-negative Ricci curvature.}  If $M$ has non-negative Ricci curvature then it is well known that $\nabla \Delta^{-1/2}$ is bounded on $L^p(M)$ for all $p \in (1, \infty)$ (cf. \cite{Bakry}). 

\medskip
- {\it Ricci curvature bounded from below.} In this case,  the local Riesz transform  $\nabla (I+\Delta)^{-1/2}$ is bounded on $L^p(M)$ for all $p \in (1, \infty)$ (cf. \cite{Bakry}). It then follows from Section \ref{sec3} that the local Littlewood-Paley-Stein functional 
$$ H^{loc}(f) := \Big(  \int_0^1 | \nabla e^{-tL} f |^2 dt \Big)^{1/2} $$
is bounded on $L^p(M)$ for all $p \in (1, \infty)$. By Proposition~\ref{prop-Q}, the lower bound (for some $C_q > 0$) 
\begin{equation}\label{eq5-1}
C_q\, \|f \|_q \le \| e^{-\Delta} f \|_q + \Big\| \Big(  \int_0^1 | \nabla e^{-tL} f |^2 dt \Big)^{1/2} \Big\|_q
\end{equation}
holds for all $q \in (1,\infty)$. Note that \eqref{eq5-1} holds for $q \in [2, \infty)$ on any Riemannian manifold since $H_\nabla$ (and hence $H^{loc}$) is always bounded on $L^p(M)$ for $p \in (1, 2]$. 

\medskip
- {\it Manifolds with doubling and Gaussian bound.} Recall that if $M$ satisfies  \eqref{doubl2} and  \eqref{gauss} then the Riesz transform is bounded on $L^p(M)$ for all $p \in (1, 2]$ and it is weak type $(1,1)$ (cf. \cite{CoulhonDuong}). 
The case $p > 2$ is more complicate and there are counter-examples (see \cite{CoulhonDuong}). One needs extra assumptions on $M$ in order to have the Riesz transform bounded on $L^p(M)$ for $p > 2$. See for example
\cite{ACDH, Carron, ChenMagOuh, CoulhonCPAM, Guillarmou-Hassell} and the references therein. We recall the following result from \cite{ChenMagOuh}.
\begin{theorem}[\cite{ChenMagOuh}, Theorem~4.1]
Let $M$ be a complete Riemannian manifold with the doubling property \eqref{doubl2} and the Gaussian upper estimate \eqref{gauss}. Suppose  that the negative part $R^-$ of the Ricci curvature satisfies the following estimate 
\begin{equation}
\int_0^1 \Big\| \frac{|R^-|^{1/2}}{v(.,t^{1/2})^{1/{r_1}} }\Big\|_{r_1} \frac{dt}{t^{1/2}} + \int_1^\infty \Big\| \frac{|R^-|^{1/2}}{v(.,t^{1/2})^{1/{r_2}} }\Big\|_{r_2} \frac{dt}{t^{1/2}} < \infty
\end{equation} for some $r_1 > 2, r_2 > 3$. Then the Riesz transform is bounded on $L^p(M)$ for $p \in (1,r_2)$.
 \end{theorem}
As a consequence, Theorems~\ref{thm-hol} and \ref{thm-mult} apply to $L = \Delta$ on $L^p(M)$ for $p \in (2,r_2)$. 

\medskip
- $M_n = \RR^n \# \RR^n$ with $n \ge 2$.  We consider the manifold $M_n$ which  consists of two copies of $\RR^n \backslash B(0,1)$ endowed with the euclidean metrics and smoothly glued along the unit balls. It is proved in \cite{CoulhonDuong} that for such manifold, the  Riesz transform is unbounded on $L^p(M_n)$ for $p > n$. We also refer to  \cite{Carron} and \cite{Guillarmou-Hassell} for more  general and  precise results.  In particular, it is proved in \cite{Carron} that the Riesz transform is bounded on $L^p(M_n)$ for $p \in (1, n)$ and this is sharp. Therefore, the  Littlewood-Paley-Stein estimates of Section \ref{sec2} are satisfied on $L^p(M_n)$ for $p \in (1, n)$. Similarly to the Riesz transform, this interval is optimal in the sense that $H_\nabla$ in \eqref{H-Ga}  is unbounded on $L^p(M_n)$ for $p > n$. In order to see this, recall that  $M_n$ satisfies the global Sobolev inequality
\begin{equation}\label{SobolevGlobal}
|f(x) - f(y)| \leq C\, d(x,y)^{1-n/p} \| \nabla f\|_p. 
\end{equation}
It is also known that there exist positive constants $c$ and $C$ such that $c r^n \leq v(x,r) \leq C r^n$ uniformly in $r > 0$ and $x \in M_n$. 
Now, if the Littlewood-Paley-Stein functional is bounded on $L^p(M)$ for some $p > n$, then it follows from Proposition~\ref{prop-multiplication} and the analyticity of the semigroup that
$$\| \nabla e^{-t\Delta}  f \|_p \leq \frac{C}{t^{1/2}} \|f\|_p. $$ 
We apply this  inequality with $f = p_t(.,y)$ (the heat kernel associated with $\Delta$)  and notice that  $p_t(.,y) = e^{-\frac{t}{2} \Delta} p_{\frac{t}{2}}(.,y)$ and then  combine this with \eqref{SobolevGlobal} to obtain \begin{align*}
|p_t(x,y) - p_t(y,y)| &\leq Cd(x,y)^{1-n/p}\| \nabla p_t(.,y) \|_p \\
&\leq Cd(x,y)^{1-n/p} \frac{t^{-1/2 + n/2p}}{v(y,t^{1/2})}.
\end{align*}
A well known  chain argument allows to obtain from  this inequality a Gaussian lower bound 
$$p_t(x,y) \geq C \frac{e^{-c\frac{d(x,y)^2}{t}}}{v(y,t^{1/2})}.$$
This lower bound is not true for $M_n$. We refer to  \cite{CoulhonDuong} for the details and additional information.\\
Note that $M_n$ has Ricci curvature bounded from below. Therefore, the local  Littlewood-Paley-Stein functional is bounded on $L^p(M_n)$ for all $p \in (1, \infty)$. It is then the Littlewood-Paley-Stein at infinity which is not bounded on $L^p(M_n)$ for $p > n$.

\medskip
\noindent\underline{Schr\"odinger operators.} 

\medskip
- {\it Potentials in the Reverse H\"older class.}   We consider $L = \Delta + V$ on $L^p(\RR^n)$ for some $n \ge 3$.  We assume that  the non-negative potential $V$ belongs to the Reverse H\"older class $B_q$, that is, there exists a constant $C > 0$ such that for every ball $B$ in $\mathbb{R}^n$, 
\begin{equation}\label{ReverseHolder}
\frac{1}{|B|}\int_B V^q dx \leq C \Big(\frac{1}{|B|}\int_B V dx\Big)^q.
\end{equation}  
It is known that this property self-improves in the sense that  there exists $\epsilon > 0$ such that $V \in B_{q+\epsilon}$. It is proved in  \cite{Shen} that if $V \in B_q$ for some $n/2 \leq q < n$, then the Riesz transform $\nabla L^{-1/2}$ is bounded on $L^p(\RR^n)$ for $1 < p \leq p_0$ where $\frac{1}{p_0} = \frac{1}{q} - \frac{1}{n}$. This result was improved in \cite{BenAli} by considering the cases $n < 3$ or $q \geq n$ and also the boundedness of $\sqrt{V} L^{-1/2}$. More precisely, it is proved in \cite{BenAli} that
\begin{theorem}\label{thm-AbA}
1- If  $V \in B_q$ for some $q > 1$, then $\nabla L^{-1/2}$ and $\sqrt{V} L^{-1/2}$ are bounded on $L^p(\RR^n)$ for 
$p \in (1, 2 (q + \epsilon))$.\\
2- If $V \in B_q$ for some $q \ge n/2$ and $q > 1$, then $\nabla L^{-1/2}$ is bounded on $L^p(\RR^n)$ for  all $p \in (1, q^* + \epsilon)$ if $q < n$ and for all $p \in (1,\infty)$ if $q \ge n$. Here $q^* = \frac{nq}{n-q}$.
\end{theorem}
We apply Theorems~\ref{thm-hol} and \ref{thm-mult} to obtain general  Littlewood-Paley-Stein estimates on $L^p(\RR^n)$ for 
$p$ in one of  intervals given in Theorem~\ref{thm-AbA}. Their reverse inequalities proved in Section \ref{sec4} hold on the dual space.

It is also proved in \cite{Shen} that the above range is optimal for the boundedness of the Riesz transform. One may then ask whether this range is optimal for the boundedness of the Littlewood-Paley-Stein functional as well. This is indeed the case. 
\begin{proposition}\label{opt-V}
There exists  $V \in B_q$ with  $n/2 \leq q < n$ such that the Littlewood-Paley-Stein functional $H$ is not bounded  on $L^p(\RR^n)$ for any 
$p > q^* + \epsilon$.
\end{proposition}
\begin{proof} We follow exactly the same arguments as in \cite[Section 7]{Shen}. Let $q_0 > n/2$ and set 
$V(x):= \frac{1}{|x|^{n/q_0}}$.  Then  $V \in B_q$ for all $q <q_0$. Therefore, the Littlewood-Paley-Stein function $H$ 
 is bounded on $L^{p}(\RR^n)$ for all $p$ such that $\frac{1}{p} > \frac{1}{q_0} - \frac{1}{n}$.  We show that it is false for $p= p_0$ with 
 $\frac{1}{p_0} = \frac{1}{q_0} - \frac{1}{n}$. Let $v$ be the function defined by 
 $$v(x) := \sum_{m=0}^\infty \frac{(\frac{1}{\mu})^{2m}|x|^{\mu m}}{m! \Gamma(\frac{n-2}{\mu}+m+1)}$$
 with $\mu = 2 - \frac{n}{q_0}$.  One has  by a direct computation 
 $$ \Delta v + V v  = 0.$$
Set $u:= \phi v$ where $\phi$ is a smooth non-negative compactly supported function with $\phi(x) = 1$ if $|x| \leq 1$. We have $$ \Delta u + V u = -2 \nabla \phi . \nabla v + v \Delta \phi.$$ 
Set  $ g := \Delta u + V u = -2 \nabla \phi . \nabla v + v \Delta \phi$.  Suppose that $H$ is bounded on $L^{p_0}(\RR^n)$. Then
$$ \|\nabla f\|_{p_0}^2  \leq C\, \|f\|_{p_0} \| L f\|_{p_0}$$ 
by Proposition~\ref{prop-multiplication}.  Therefore,
\begin{align*}
\| \nabla u \|_{p_0}^2  &= \|\nabla L^{-1} g \|_{p_0}^2 \\
&\le C\,  \| g \|_{p_0}  \| u \|_{p_0} < \infty
\end{align*}
since $u$ and $g$ are  in $L^{p_0}$ (they are bounded and compactly supported).  But  $\nabla u$ is not in $L^{p_0}(\RR^n)$  because $|\nabla u| \sim \frac{1}{|x|^{n/{p_0}}}$ as $x \sim 0$.  \end{proof}

\medskip
- {\it Schr\"odinger operators on manifolds.}  Riesz transforms associated with Schr\"odinger operators have been also studied on Riemannian manifolds $M$. As we already mentioned before, if $M$ satisfies \eqref{doubl} and \eqref{gauss} then $\nabla L^{-1/2}$ and $\sqrt{V} L^{-1/2}$ are bounded on $L^p(M)$ for $p \in (1,2]$. Here the potential $V$ is non-negative and locally integrable. See \cite{DOY} where this is stated on $\RR^N$ but the proof works on manifolds having \eqref{doubl} and \eqref{gauss}. \\
The case $p > 2$ is again  complicate (even if $M = \RR^N$). We recall the following result which deals also with potentials having a non-trivial negative part.  
\begin{theorem}[\cite{AO}, Theorem 3.9]
Suppose that  $M$ satisfies \eqref{doubl2} and  \eqref{gauss}. Suppose in addition that $V^-$ is subcritical with rate $\alpha \in (0, 1)$, i.e.,  for all suitable $f$ in $L^2(M)$ we have
\begin{equation}\label{SC}
\int_{M} V^- f^2 dx \leq \alpha \int_{M} (|\nabla f|^2 + V^+ f^2 )dx.
\end{equation}Assume there exist $r_1, r_2 > 2$ such that 
\begin{equation}\label{integ}
\int_0^1 \Big\| \frac{|V|^{1/2}}{v(.,s^{1/2})^{1/{r_1}} }\Big\|_{r_1} \frac{ds}{s^{1/2}} + \int_1^\infty \Big\| \frac{|V|^{1/2}}{v(.,s^{1/2})^{1/{r_2}} }\Big\|_{r_2} \frac{ds}{s^{1/2}} < \infty.
\end{equation}
Let $r=  \inf(r_1,r_2)$. If $N \leq 2$, then the operators $\Delta^{1/2} L^{-1/2}$ and $V^{1/2} L^{-1/2}$ are bounded on $L^p$ for $p\in (1,r)$. If $N > 2$, the same operators are bounded on $L^p$ for $p \in (p_0', \frac{p_0 r}{p_0 +r})$ where $p_0 = \frac{N}{N-2}\frac{2}{1-\sqrt{1-\alpha}}$. In particular, if the Riesz transform $\nabla \Delta^{-1/2}$  is bounded on $L^p$ with $p$ in this range, then $\nabla L^{-1/2}$ is also bounded. 
\end{theorem}
\noindent See also \cite{Devyver} for related results and additional information.

The above integrability condition in \eqref{integ} gives then a range of $p$'s for which the Littlewood-Paley-Stein functionals are bounded on $L^p(M)$. Finally, we mention the following negative result (see \cite{Ouhabaz} in the case $M = \RR^N$). 

\begin{proposition} Assume that $M$ satisfies \eqref{doubl2},  \eqref{gauss} and  the local Sobolev inequality 
$$|f(x) - f(x')| \leq C_{x,x'} d(x,x')^{1-N/p} \|\nabla f\|_p.$$ 
Suppose also that  there exists a positive bounded function $\psi$ such that $e^{-tL} \psi = \psi$ for all $t  \geq 0$. If the Littlewood-Paley-Stein functional 
$$ H_\nabla(f) =  \Big(  \int_0^\infty | \nabla e^{-tL} f |^2 dt \Big)^{1/2}$$
is bounded on $L^p(M)$ for some $p > \max(N,2)$, then $V = 0$. Here $N$ is the constant from the doubling condition \eqref{doubl2}.
\end{proposition}

\begin{proof}
Assume that $H$ is bounded on $L^p$, then for suitable $f$,
$$\| \nabla f\|_p \leq C\, \| f \|^{1/2}_p \| L f\|^{1/2}_p.$$ 
Taking $f = e^{-tL} g$ and using the analyticity of the semigroup we obtain for all $g \in L^p$ $$ \|\nabla e^{-tL} g\|_p \leq \frac{C}{t^{1/2}} \|g \|_p.$$ 
We conclude using Theorem~6.1 in \cite{ChenMagOuh}. Note that in this  reference, it is assumed  that $M$ satisfies Poincar\'e inequalities, which in turn imply the above local Sobolev inequality. It is this later inequality which is used in the proof there.
\end{proof}

\section{Elliptic operators on domains}\label{sec6}
We have chosen to write the previous sections  in the framework of Schr\"odinger operators on manifolds. The results remain valid for elliptic operators with real bounded measurable coefficients and subject to Dirichlet boundary conditions on a domain. The proofs, after a little adaptation,  are  the same. 

Let $\Omega$ be an open subset of $\RR^N$ ($N \ge 1$). We consider for $k,l \in \{1,...,N\}$ bounded measurable functions
$ a_{kl} = a_{lk}: \Omega \to \RR$. We suppose the usual ellipticity condition
\[
\sum_{k,l=1}^N a_{kl}(x) \xi_k \xi_l \ge \nu | \xi |^2
\]
for all $\xi = (\xi_1,...,\xi_n) \in \RR^N$, where $\nu > 0$ is a constant independent of $x$. Set $A(x) := (a_{kl}(x))_{1\le k,l \le N}$. We define the elliptic operator 
$L = - div(A(x) \nabla \cdot)$ with Dirichlet boundary conditions. It is the operator associated with the symmetric form
\[
{\gothic a}(u,v) = \sum_{k,l=1}^N \int_\Omega a_{kl} \partial_k u \partial_l v, \quad u, v \in W^{1,2}_0(\Omega).
\]
It is  known that the heat kernel of $L$ satisfies a Gaussian upper bound and the Riesz transform $\nabla L^{-1/2}$ is bounded on 
$L^p(\Omega)$ for all $p \in (1, 2]$. In addition, $L$ satisfies spectral multiplier theorems. See \cite[Chapters VI and VII]{Ouh05}. The fact that $\Omega$, endowed with the Euclidean distance and Lebesgue measure,  may not satisfy the doubling property \eqref{doubl}\footnote{except if $\Omega$ is bounded and has smooth boundary, Lipschitz boundary is enough.} can be bypassed in the proofs of the boundedness of the Riesz transform and spectral multipliers. \\
Thus, the Littlewood-Paley-Stein estimates \eqref{eq2-1} and \eqref{eq2-3} hold on $L^p(\Omega)$ for all $p \in (1, 2]$. More precisely,
\begin{equation}\label{eq6-1}
\Big\|  \Big( \sum_k \int_0^\infty | \nabla  m_k(L) F(tL) f_k |^2 dt \Big)^{1/2} \Big\|_{L^p(\Omega)} \le C\,  \sup_{k} \| m_k \|_{H^\infty(\Sigma(\theta))}\,  \Big\| \Big( \sum_k |f_k|^2 \Big)^{1/2} \Big\|_{L^p(\Omega)}
\end{equation}
for bounded holomorphic functions $m_k$ and $F$ on a sector of angle $\theta > 0$\footnote{here we may take any $\theta > 0$ and not necessarily $\omega_p$ as in Theorem~\ref{thm-hol}. The reason is that the Gaussian upper bound implies the existence of a bounded holomorphic functional with angle $\theta$. This follows readily from the fact that $L$ satisfies spectral multiplier theorems.}. If the functions $m_k$ are supported in $[\frac{1}{2},2]$ and belong to the Sobolev space $W^{\delta, 2}(\RR)$ for some $\delta  > N | \frac{1}{2}-\frac{1}{p}| + \frac{3}{2}$, then
\begin{equation}\label{eq6-2}
 \Big\|  \Big( \sum_k \int_0^\infty | \nabla  m_k(tL) f_k |^2 dt \Big)^{1/2} \Big\|_{L^p(\Omega)} 
 \le C\,  \sup_{k} \| m_k \|_{W^{\delta,2}}\,  \Big\| \Big( \sum_k |f_k|^2 \Big)^{1/2} \Big\|_{L^p(\Omega)}.
\end{equation}
From this and little modifications in the proofs in Section \ref{sec4} we obtain the lower bounds on $L^q(\Omega) \cap L^2(\Omega)$ for all $q \in [2, \infty)$. In particular,
\begin{equation}\label{eq6-3}
C\, \| f \|_{L^q(\Omega)} \le  \Big\| \Big( \int_0^\infty | \nabla e^{-tL} f |^2 dt \Big)^{1/2} \Big\|_{L^q(\Omega)},
\end{equation}
and
\begin{equation}\label{eq6-4}
C\, \| f \|_{L^q(\Omega)} \le \| e^{-L} f \|_{L^q(\Omega)} + \Big\| \Big( \int_0^1 | \nabla e^{-tL} f |^2 dt \Big)^{1/2} \Big\|_{L^q(\Omega)}.
\end{equation}
It is remarkable that no regularity assumption is required  on the domain nor on the coefficients of the operator. For another proof of \eqref{eq6-4} and related inequalities on a smooth domain, we refer to \cite{Planchon}. 

In the next result we show that  if $\Omega = \RR^N$, the previous lower bounds  hold for all $q \in (1, \infty)$.
\begin{proposition}\label{rev-all-q}
Let  $L = - div(A(x) \nabla \cdot)$ be a self-adjoint elliptic operator  with real bounded measurable coefficients $a_{kl}$. Then for all $q \in (1, \infty)$
\begin{equation}\label{eq6-5}
C\, \| f \|_{L^q(\RR^N)} \le  \Big\| \Big( \int_0^\infty | \nabla e^{-tL} f |^2 dt \Big)^{1/2} \Big\|_{L^q(\RR^N)}
\end{equation}
and
\begin{equation}\label{eq6-6}
C\, \| f \|_{L^q(\RR^N)} \le \| e^{-L} f \|_{L^q(\RR^N)} + \Big\| \Big( \int_0^1 | \nabla e^{-tL} f |^2 dt \Big)^{1/2} \Big\|_{L^q(\RR^N)}.
\end{equation}
\end{proposition}
\begin{proof} Because of \eqref{eq6-3} and \eqref{eq6-4} we  consider the case  $q \in (1,2)$ only. \\
Since the semigroup $e^{-tL}$ is sub-Markovian, $L$  has a bounded  holomorphic functional calculus on $L^p(\RR^N)$. Therefore,  it has bounded square functions on $L^p(\RR^N)$ for all $p \in (1, \infty)$. A standard duality argument gives then (for $q \in (1, \infty)$)
\begin{equation}\label{eq6-7}
C\, \|f\|_{L^q(\RR^N)} \le \Big\| \Big( \int_0^\infty | L^{1/2} e^{-tL} f |^2 dt \Big)^{1/2} \Big\|_{L^q(\RR^N)}.
\end{equation}
On the other hand, it follows from \cite[Theorem~2, p.115]{AuscherAsterisque} that there exists a Calder\'on-Zygmund operator $U$ such that  $L^{1/2}f = U \nabla f$. Therefore,
\[
C\, \|f\|_{L^q(\RR^N)} \le \Big\| \Big( \int_0^\infty | U \nabla e^{-tL} f |^2 dt \Big)^{1/2} \Big\|_{L^q(\RR^N)}
\]
The operator $U$ is bounded on $L^q(\RR^N)$. We then apply the same strategy of proof as in Theorem~\ref{thm-hol} and use the Kahane inequality to bound the RHS term by 
\[
C'\, \Big\| \Big( \int_0^\infty | \nabla e^{-tL} f |^2 dt \Big)^{1/2} \Big\|_{L^q(\RR^N)}.
\]
This proves \eqref{eq6-5}. \\
In order to prove \eqref{eq6-6} for $q \in (1,2)$, we write
\begin{eqnarray*}
&&\Big\| \Big( \int_0^\infty | \nabla e^{-tL} f |^2 dt \Big)^{1/2} \Big\|_{L^q(\RR^N)}\\
&&\le \Big\| \Big( \int_0^1 | \nabla e^{-tL} f |^2 dt \Big)^{1/2} \Big\|_{L^q(\RR^N)} + \Big\| \Big( \int_1^\infty | \nabla e^{-tL} f |^2 dt \Big)^{1/2} \Big\|_{L^q(\RR^N)}\\
&& =  \Big\| \Big( \int_0^1 | \nabla e^{-tL} f |^2 dt \Big)^{1/2} \Big\|_{L^q(\RR^N)} + \Big\| \Big( \int_0^\infty | \nabla e^{-tL} e^{-L}f |^2 dt \Big)^{1/2} \Big\|_{L^q(\RR^N)}\\
&& \le \Big\| \Big( \int_0^1 | \nabla e^{-tL} f |^2 dt \Big)^{1/2} \Big\|_{L^q(\RR^N)} + C''\, \| e^{-L} f \|_p.
\end{eqnarray*}
Note that in the last inequality we use the boundedness of the Littlewood-Paley-Stein functional on $L^q(\RR^N)$ for $q \in (1,2)$. Now, \eqref{eq6-6} follows from \eqref{eq6-5} and the previous inequality. 
\end{proof}
We finish this section by mentioning another sort of Littlewood-Paley-Stein functionals, called  {\it conical vertical square functions},  and defined by
\[
S f(x) := \Big( \int_0^\infty \int_{|x-y| < \sqrt{t}} | \nabla_y e^{-tL} f(y) |^2 \frac{dy\, dt}{t^{N/2}} \Big)^{1/2}.
\]
It is proved in \cite{Auscher-Hofmann}, among other things,  that $S$ is bounded on $L^p(\RR^N)$ for all 
$p \in (1,\infty)$.  Thus, the  functionals  $S$ and $H$ have different behavior on $L^p(\RR^N)$ for $p > 2$.  It is of interest to  study the corresponding conical functionals $S$ for Schr\"odinger operators on manifolds or for elliptic operators on arbitrary domains of $\RR^N$. This will be done  in a forthcoming  project. 


\end{document}